\documentclass[12pt]{amsart}
\usepackage{amsfonts, amssymb}
\usepackage{amsmath}
\usepackage{amsthm}
\usepackage{cite}
\numberwithin{equation}{section}

\newtheorem{theorem}{Theorem}[section]
\newtheorem{lemma}[theorem]{Lemma}

\begin{document}

\title[Bank--Laine functions]{Zero distribution of finite order Bank--Laine functions}

\author{Yueyang Zhang}
\address{School of Mathematics and Physics, University of Science and Technology Beijing, No.~30 Xueyuan Road, Haidian, Beijing, 100083, P.R. China}
\email{zhangyueyang@ustb.edu.cn}
\thanks{The author is supported by the Fundamental Research Funds for the Central Universities~{(FRF-TP-19-055A1)} and a Project supported by the National Natural Science Foundation of China~{(12301091)}}



\date{\today}

\commby{}

\begin{abstract}
It is known that a Bank--Laine function $E$ is a product of two normalized solutions of the second order differential equation $f''+Af=0$ $(\dag)$, where $A=A(z)$ is an entire function. By using Bergweiler and Eremenko's method of constructing transcendental entire function $A(z)$ by gluing certain meromorphic functions with infinitely many times, we show that, for each $\lambda\in[1,\infty)$ and each $\delta\in[0,1]$, there exists a Bank--Laine function $E$ such that $E=f_1f_2$ with $f_1$ and $f_2$ being two entire functions such that $\lambda(f_1)=\delta\lambda$ and $\lambda(f_2)=\lambda$, respectively. We actually provide a complete construction of the Bank--Laine functions given by Bergweiler and Eremenko.

\end{abstract}

\maketitle


\section{Introduction}\label{intro} 

An entire function $E(z)$ is called a \emph{Bank--Laine function} if $E(z)=0$ implies that $E'(z)\in\{-1,1\}$. In particular, $E(z)$ is called \emph{special} if $E(z)=0$ implies that $E'(z)=1$. It is known that every Bank--Laine function $E$ is written as $E=f_1f_2$, where $f_1$ and $f_2$ are two normalized solutions of the second order differential equation
\begin{equation}\label{bank-laine0}
f''+A(z)f=0,
\end{equation}
where $A(z)$ is an entire function; see \cite[Proposition~6.4]{Laine1993}. All solutions of equation \eqref{bank-laine0} are entire functions. A pair $(f_1,f_2)$ of solutions of \eqref{bank-laine0} is called \emph{normalized} if $f_1$ and $f_2$ are linearly independent and the \emph{Wronskian determinant} $W(f_1,f_2)=f_1f_2'-f_1'f_2$, which is a nonzero constant, is equal to $1$. Denote $F=f_2/f_1$. Then $F$ is a locally univalent meromorphic function. The quotient $F/F'=f_1f_2$ is a Bank--Laine function, and all Bank--Laine functions arise in this way. Note that a normalized pair $(f_1,f_2)$ can be recovered from $F$ by the formulas $f_1^2=1/F'$ and $f_2^2=F^2/F'$. Zeros of a Bank--Laine function $E$ are zeros and poles of $F$, and $E$ is special if and only if $F$ is entire. See~\cite{Laine1993} for the basics about the second order differential equation \eqref{bank-laine0}.

Bank and Laine~\cite{Banklaine1982,Banklaine1982-2} initiated the study on the complex oscillation of second order linear differential equation \eqref{bank-laine0} in the framework of Nevanlinna theory. See \cite{Hayman1964Meromorphic,Laine1993} for the basic notation and fundamental results of Nevanlinna theory. Recall that the \emph{order} of an entire function $f$ is defined as
\begin{equation*}
\rho(f)=\limsup_{r\to\infty}\frac{\log\log M(r,f)}{\log r},
\end{equation*}
where $M(r,f)$ is the maximum modulus of $f$ on the circle $|z|=r$. When $A$ is transcendental, an application of the lemma on the logarithmic derivative easily yields that all nontrivial solutions of \eqref{bank-laine0} satisfy $\rho(f)=\infty$. Also recall that the \emph{exponent of convergence} of zeros of an entire function $f$ is defined as
\begin{equation*}
\lambda(f)=\limsup_{r\to\infty}\frac{\log n(r,0,f)}{\log r},
\end{equation*}
where $n(r,0,f)$ denotes the number of zeros of $f$ in the disc $\{z: |z|\leq r\}$. Bank and Laine \cite{Banklaine1982,Banklaine1982-2} proved: For two normalized solutions $f_1$ and $f_2$ of equation \eqref{bank-laine0} and $E=f_1f_2$, if $\rho(A)\not\in \mathbb{N}$, then $\lambda(E) \geq \rho(A)$; if $\rho(A)<1/2$, then $\lambda(E)=\infty$. Later, Shen~\cite{Shen1985} and Rossi~\cite{Rossi1986} relaxed the condition $\rho(A)<1/2$ to the case $\rho(A)=1/2$ independently, where they actually showed that if $1/2\leq \rho(A)<1$, then the conclusion $\lambda(E) \geq \rho(A)$ can be improved to satisfy the inequality
\begin{equation}\label{bank-laine1}
\frac{1}{\rho(A)}+\frac{1}{\lambda(E)}\leq 2.
\end{equation}
Based on these results, Bank and Laine conjectured that $\max\{\lambda(f_1),\lambda(f_2)\}=\infty$ whenever $\rho(A)\not\in \mathbb{N}$; see also \cite[Chapter~4]{Laine1993}. \emph{The Bank--Laine conjecture} has attracted much interest; see the surveys \cite{Gundersen2014,lainetohge2008} and references therein. However, Bergweiler and Eremenko \cite{Bergweilereremenko2017,Bergweilereremenko2019} disproved this conjecture by constructing counterexamples for the coefficient $A$ such that $\rho(A)$ is not an integer and equation \eqref{bank-laine0} has two linearly independent solutions such that $\lambda(f_1)<\infty$ and $f_2$ has no zeros.

All Bank--Laine functions have order of growth at least $1$. Bergweiler and Eremenko's construction shows for the first time that there exists a special Bank--Laine function $E$ such that $\lambda(E)=\lambda$ for any $\lambda\in[1,\infty)$. Though the construction is very sophisticated, Bergweiler and Eremenko have started from a family of locally univalent meromorphic functions $F=f_2/f_1$ with $f_1$ and $f_2$ being two normalized solutions such that $\lambda(f_i)<\infty$, $i=1,2$, of the second order differential equation
\begin{equation}\label{Bank-laine se3simple}
f''-\left(\frac{1}{4}e^{2z}+\frac{1}{2}b_1e^{z}+b_2\right)f=0,
\end{equation}
where $b_1\not=0$ and $b_2$ are two constants. Indeed, all nontrivial solutions of the equation $f''+(e^{z}-K)f=0$ such that $\lambda(f)<\infty$ have been clearly characterized in~\cite{Banklaine1983-1}; see also \cite[Theorem~5.22]{Laine1993}. In \cite{zhang2021,zhang2021-1,Zhang2022}, the present author solves nontrivial solutions such that $\lambda(f)<\infty$ of equation \eqref{Bank-laine se3simple} and, in particular, proved the following

\begin{theorem}[\cite{Zhang2022}]\label{maintheorem1}
Let $b_1\not=0$ and $b_2$ be two constants. Suppose that~\eqref{Bank-laine se3simple} has a nontrivial solution~$f$ such that $\lambda(f)<\infty$. Then $f=\kappa e^{h}$, $\kappa=\sum_{i=0}^ka_ie^{iz}$ and $h=c_0e^{z}+cz$, where $k\geq 0$ is an integer, $c_0$ and $c$ are constants such that $c_0^2=1/4$, $2c_0(c+k)+c_0=b_1/2$ and $c^2=b_2$, and $a_0$, $\cdots$, $a_k$ are constants such that $a_0a_k\not=0$ and
\begin{equation*}
\begin{split}
2c_0(k+1-i)a_{i-1}=(2ic+i^2)a_i, \qquad i=1,\cdots,k.
\end{split}
\end{equation*}

\end{theorem}

Note that equation \eqref{Bank-laine se3simple} takes a slightly different form from that in \cite{Zhang2022}. By Theorem~\ref{maintheorem1} we give a different formulation from the results in~\cite[Theorem~1.6]{ChiangIsmail2006} which states that: Equation \eqref{Bank-laine se3simple} admits two linearly independent solutions $f_1$ and $f_2$ such that $\max\{\lambda(f_1),\lambda(f_2)\}<\infty$ if and only if there are two distinct nonnegative integers $k_1,k_2$ such that $b_1=\pm(k_1-k_2)$ and $4b_2=4c^2=(k_1+k_2+1)^2$. In particular, it is possible that $\min\{\lambda(f_1),\lambda(f_2)\}=0$. Letting $c=-(k_1+k_2+1)/2$, then equation \eqref{Bank-laine se3simple} takes the form
\begin{equation}\label{Bank-laine se3simpletr}
f''-\left[\frac{1}{4}e^{2z}+\frac{1}{2}(k_2-k_1)e^{z}+\frac{(k_1+k_2+1)^2}{4}\right]f=0.
\end{equation}
Choosing $c_0=-1/2$, we may write the two linearly independent solutions $f_1$ and $f_2$ as
\begin{equation*}
\begin{split}
f_1=\kappa_1e^{-\frac{k_1+k_2+1}{2}z}e^{-\frac{1}{2}e^{z}}=\left(\sum_{i=0}^{k_1}a_ie^{iz}\right)e^{-\frac{k_1+k_2+1}{2}z}e^{-\frac{1}{2}e^{z}},
\end{split}
\end{equation*}
where
\begin{equation*}
\begin{split}
-(k_1+1-i)a_{i-1}=(2ic+i^2)a_i, \qquad i=1,\cdots,k_1
\end{split}
\end{equation*}
and
\begin{equation*}
\begin{split}
f_2=\kappa_2e^{-\frac{k_1+k_2+1}{2}z}e^{\frac{1}{2}e^{z}}=\left(\sum_{j=0}^{k_2}b_je^{jz}\right)e^{-\frac{k_1+k_2+1}{2}z}e^{\frac{1}{2}e^{z}},
\end{split}
\end{equation*}
where
\begin{equation*}
\begin{split}
(k_2+1-j)b_{j-1}=(2jc+j^2)b_j, \qquad j=1,\cdots,k_2.
\end{split}
\end{equation*}
Suppose that $(f_1,f_2)$ is a normalized pair of solutions of \eqref{Bank-laine se3simpletr}. It is easy to see that $a_0$ and $b_0$ should satisfy $a_{k_1}b_{k_2}=\frac{k_1!k_2!}{[(k_1+k_2)!]^2}a_0b_0=1$ and a simple calculation shows that
\begin{equation*}
\left\{
  \begin{array}{ll}
    a_i=\frac{k_2!}{(i+k_2)!}a_0,& \qquad  i=1,\cdots,k_1,\\
    b_j=\frac{(-1)^jk_1!}{(j+k_1)!}b_0,& \qquad j=1,\cdots,k_2.
  \end{array}
\right.
\end{equation*}
Let
\begin{equation*}
\begin{split}
F=\frac{f_2}{f_1}=\left(\frac{\sum_{j=0}^{k_2}b_je^{jz}}{\sum_{i=0}^{k_1}a_ie^{iz}}\right)e^{e^{z}}.
\end{split}
\end{equation*}
We choose $a_0=b_0$. Let $A_0=B_0=1$ and $A_i=a_i/a_0$ and $B_j=b_j/b_0$. For every pair $(m,n)$ of two non-negative integers $m$ and $n$, we consider the rational function
\begin{equation}\label{recuree1pm8}
\begin{split}
R_{m,n}(z)=\frac{P_n(z)}{Q_m(z)}=\frac{\sum_{j=0}^{2n}B_jz^j}{\sum_{i=0}^{m}A_iz^i}.
\end{split}
\end{equation}
Thus we may write
\begin{equation*}
\begin{split}
h_{m,n}(z)=R_{m,n}(z)e^{z}=\left(\frac{\sum_{j=0}^{2n}B_jz^j}{\sum_{i=0}^{m}A_iz^i}\right)e^{z}.
\end{split}
\end{equation*}
It follows that $A_{m}B_{2n}=\frac{m!(2n)!}{[(m+2n)!]^2}=\frac{1}{\binom{m+2n}{m}(m+2n)!}$.
Then the meromorphic (entire when $m=0$) function
\begin{equation}\label{recuree1pm9}
\begin{split}
g_{m,n}(z)=h_{m,n}(e^z)=\left(\frac{\sum_{j=0}^{2n}B_je^{jz}}{\sum_{i=0}^{m}A_ie^{iz}}\right)e^{e^{z}}
\end{split}
\end{equation}
satisfies
\begin{equation}\label{recuree1pm10}
\begin{split}
g'_{m,n}(z)=\frac{1}{\binom{m+2n}{m}(m+2n)!}\frac{e^{e^{z}+(m+2n+1)z}}{\left(\sum_{i=0}^{m}A_ie^{iz}\right)^2}.
\end{split}
\end{equation}
Since $A_j$ are all positive constants, $g'_{m,n}(z)\not=0$ for all $z\in \mathbb{C}$ and $g_{m,n}$ is increasing on $\mathbb{R}$, and satisfies $g_{m,n}(x)\to 1$ as $x\to -\infty$ as well as $g_{m,n}(x)\to +\infty$ as $x\to +\infty$. Bergweiler and Eremenko's construction starts from the functions $g_{m,n}(z)$ in \eqref{recuree1pm9} with $m=0$ and two or infinitely many different even numbers $2n$. We shall consider the generic choice of $m$, $n$ and prove the following

\begin{theorem}\label{maintheorem2}
For every $\lambda\in[1,\infty)$ and every $\delta\in[0,1]$, there exists a Bank--Laine function $E$ such that  $\lambda(E)=\rho(E)=\lambda$ and $E=f_1f_2$ for two entire functions $f_1$ and $f_2$ such that $\lambda(f_1)=\delta\lambda$ and $\lambda(f_2)=\lambda$.
\end{theorem}

In section~\ref{Proof of maintheorem2} we actually follow the process in \cite{Bergweilereremenko2019} to prove the equivalent statement: Let $\rho\in(1/2,1]$ and $\gamma=1/(2\rho-1)$ and $\delta\in[0,1]$. Then there exists an entire function $A$ of order $\rho(A)=\rho$ such that the differential equation \eqref{bank-laine0} has two linearly independent solutions $f_1$ and $f_2$ such that $E=f_1f_2$ satisfies $\lambda(E)=\rho\gamma$ and $f_1$ and $f_2$ satisfy $\lambda(f_1)=\delta\rho\gamma$ and $\lambda(f_2)=\rho\gamma$, respectively. In particular, when $\delta=0$ the function $f_1$ can be chosen to have no zeros so that the corresponding Bank--Laine function $E$ is special.

Note that solutions of \eqref{Bank-laine se3simple} provide such $f_1$ and $f_2$ in the case $\lambda=1$ and $\delta=0$ or $\delta=1$. For each fixed $\rho\in(1/2,1]$, we see that the two numbers $\rho$ and $\lambda$ satisfy $1/\rho+1/\lambda=2$. Thus, for each $\lambda\in[1,\infty)$, we may choose $\rho$ and $\delta$ suitably so that $f_1$ and $f_2$ have the desired properties in Theorem~\ref{maintheorem2}.

The associated \emph{Schwarzian differential equation} to each equation \eqref{bank-laine0} is given by
\begin{equation}\label{recuree1pm4}
\begin{split}
S(F):=\left(\frac{F''}{F'}\right)'-\frac{1}{2}\left(\frac{F''}{F'}\right)^2=\frac{F'''}{F'}-\frac{3}{2}\left(\frac{F''}{F'}\right)^2=2A,
\end{split}
\end{equation}
where the expression $S(F)$ is called the \emph{Schwarzian derivative} of $F$. The Schwarzian derivative is invariant under the linear fractional transformation of $F$. The general solution of \eqref{recuree1pm4} is $F=f_2/f_1$.
To study the asymptotic behaviors of the functions $A(z)$ and $E(z)$, Bergweiler and Eremenko have used the fact that the Schwarzian derivative $S(F)$ can be factored as $S(F)=B(F/F')/2$, where
\begin{equation*}
\begin{split}
B(E):=-2\frac{E''}{E}+\left(\frac{E'}{E}\right)^2-\frac{1}{E^2}.
\end{split}
\end{equation*}
See Bank and Laine~\cite{Banklaine1982,Banklaine1982-2}. The general solution of the differential equation $B(E)=4A$, that is, of the equation
\begin{equation}\label{recuree1pm7}
\begin{split}
4A=-2\frac{E''}{E}+\left(\frac{E'}{E}\right)^2-\frac{1}{E^2},
\end{split}
\end{equation}
is a product of a normalized pair of solutions of \eqref{bank-laine0}. We shall also use the identity \eqref{recuree1pm7} to study the asymptotic behaviors of $A(z)$ and $E(z)$ in section~\ref{Proof of maintheorem3} and prove the following

\begin{theorem}\label{maintheorem3}
The functions $A$ and $E$ before can be chosen so that
\begin{equation}\label{recuree1pm3 fu2}
\begin{split}
\frac{1}{2}\log |A(re^{i\theta})|\sim \log \frac{1}{|E(re^{i\theta})|}\sim r^{\rho}\cos(\rho \theta) \quad \text{for} \quad |\theta|\leq (1-\varepsilon)\frac{\pi}{2\rho},
\end{split}
\end{equation}
while, with $\sigma=\rho/(2\rho-1)$,
\begin{equation}\label{recuree1pm3 fu3}
\begin{split}
\log |E(-re^{i\theta})|\sim r^{\sigma}\cos(\sigma \theta) \quad \text{for} \quad |\theta|\leq (1-\varepsilon)\frac{\pi}{2\sigma},
\end{split}
\end{equation}
and
\begin{equation}\label{recuree1pm3 fu4}
\begin{split}
|A(-re^{i\theta})|\sim \frac{\sigma^2}{4}r^{2\sigma-2} \quad \text{for} \quad |\theta|\leq (1-\varepsilon)\frac{\pi}{2\sigma},
\end{split}
\end{equation}
uniformly as $r\to\infty$, for any $\varepsilon\in(0,1)$.

\end{theorem}

To prove Theorem~\ref{maintheorem2}, we first obtain a locally univalent function $F$ by (geometrically) gluing infinitely many different $g_{m,n}$s in \eqref{recuree1pm9} in the same way as that in \cite{Bergweilereremenko2019}. The process can be visualized as follows: Consider the strips
\begin{equation*}
\begin{split}
\Pi_k=\{x+iy: 2\pi(k-1)<y<2\pi k\}.
\end{split}
\end{equation*}
Let $\Im_k$ be the bordered Riemann surface spread over the plane, which is the image of this strip under the function
\begin{equation*}
\begin{split}
g_{m_k,n_k}(z)=R_{m_k,n_k}(e^z)e^{e^{z}}.
\end{split}
\end{equation*}
All these Riemann surfaces have two boundary components which project onto the ray $(1,+\infty)\subset \mathbb{R}$. By gluing them together along the real and imaginary axis, in the same order as the strips $\prod_k$ are glued together in the plane, the resulting Riemann surface $\Im$ is open and simply connected. The function $F$ in Theorems~\ref{maintheorem2} will be the conformal mapping from $\mathbb{C}$ to $\Im$. Unlike in \cite{Bergweilereremenko2019} where $g_{m_k,n_k}(z)$ has no poles (i.e., $m_k=0$ for all $k$), in the proof of Theorem~\ref{maintheorem2} we allow $g_{m_k,n_k}(z)$ to have some poles. In order to obtain two functions $f_1$ and $f_2$ with prescribed exponents of convergence of zeros, we shall choose two suitable infinite sequence $(n_k)$ and $(m_k)$. Also, for technical reasons we will actually work with the function $g_{m_k,n_k}(z+s_{m_k,n_k})$ with a carefully chosen constant $s_{m_k,n_k}$.

When $\delta=1$, the choice of the sequence $(m_k)$ implies that the number of zeros of the function $f_1$ in Theorem~\ref{maintheorem2} satisfies $d_2(\log r)^{-2}r^{\rho\gamma}\leq n(r,0,f_1)\leq d_1(\log r)^{-2}r^{\rho\gamma}$ for all large $r$, where $d_1$ and $d_2$ are two positive constants. It seems that the method in the proof of Theorem~\ref{maintheorem2} does not allow us to remove the term $(\log r)^{-2}$.
In~\cite{Zhang2023}, the present author also pastes two different $g_{m,n}$s in \eqref{recuree1pm9} along the real axis as in \cite{Bergweilereremenko2017} and obtain Bank--Laine functions having logarithmic spiral behavior. In particular, the two functions $w_1$ and $w_2$ constructed in \cite[Theorem~1.1]{Zhang2023} both satisfy $d_4 r^{\rho}\leq n(r,0,w_i)\leq d_3 r^{\rho}$, $i=1,2$, as $r\to\infty$, where $d_3$ and $d_4$ are two positive constants.

Theorem~\ref{maintheorem2} completes the construction in \cite[Corollary~1.1]{Bergweilereremenko2019}. In~\cite{Bergweilereremenko2019}, Bergweiler and Eremenko further considered a more general setting on $\rho\in(1/2,\infty)$ and used some functions obtained from $g_{0,n}$s in \eqref{recuree1pm9}, not the functions $g_{0,n}$s themselves, to obtain more special Bank--Laine functions. This makes the process of gluing infinitely many different locally univalent meromorphic functions very complicated. We may also use the function $G(z)$ defined in subsection~\ref{Definition of a quasimeromorphic map} to complete the construction in \cite[Corollary~1.2]{Bergweilereremenko2019}.

In the following, we keep the permanent notation: $(f_1,f_2)$ is a normalized pair of solutions of \eqref{bank-laine0}, the quotient $F=f_2/f_1$ is a solution of $S(F)=2A$, and $E=f_1f_2=F/F'$ is a solution of \eqref{recuree1pm7}.

\section{Proof of Theorem~\ref{maintheorem2}}\label{Proof of maintheorem2} 

\subsection{Preliminary lemmas}\label{Preliminary lemmas} 

We present some preliminarily lemmas for the proof of Theorem~\ref{maintheorem2}, the proofs for which are very similar to that in \cite{Bergweilereremenko2019}. Since now there is a nontrivial term $Q_m(z)$ in \eqref{recuree1pm8}, which is taken to be a constant in \cite{Bergweilereremenko2017,Bergweilereremenko2019}, we shall present most of the details of the proof.


\begin{lemma}\label{Lemma1}
Let $\gamma>1$. Then there exists a sequence $(n_k)$ of nonnegative integers, with $n_1=0$, such that the function $h:[0,\infty)\to[0,\infty)$ which satisfies $h(0)=0$ and which is linear on the intervals $[2\pi (k-1), 2\pi k]$ and has slope $2n_k+1$ there satisfies
\begin{equation}\label{recuree1pm15}
\begin{split}
h(x)=x^{\gamma}+O(x^{\gamma-2})+O(1), \quad x\to\infty
\end{split}
\end{equation}
and
\begin{equation}\label{recuree1pm18}
\begin{split}
2n_k+1=\gamma(2\pi k)^{\gamma-1}+O\left(k^{\gamma-2}\right)+O(1), \quad k\to\infty.
\end{split}
\end{equation}
Moreover, letting $\alpha:[0,\infty)\to[0,\infty)$ be a function such that $\alpha(x)=(\log (x+2\pi))^{-2}$ and $\mathfrak{h}:[0,\infty)\to[0,\infty)$ be a function linear and non-decreasing on each interval $[2\pi (k-1), 2\pi k]$ such that $\mathfrak{h}(0)=0$ and $\mathfrak{h}(x)=O(\alpha(x)x^{\gamma})$ as $x\to\infty$, then the function $g(x)$ defined on $[0,\infty)$ by $h(g(x))+\mathfrak{h}(g(x))=x^{\gamma}$ satisfies
\begin{equation}\label{recuree1pm16}
\begin{split}
g(x)=x+O(\alpha(x)x)+O(x^{-1})+O(x^{1-\gamma}), \quad x\to\infty
\end{split}
\end{equation}
and
\begin{equation}\label{recuree1pm17}
\begin{split}
g'(x)=1+O(\alpha(x))+O(x^{-1})+O(x^{1-\gamma}), \quad x\to\infty,
\end{split}
\end{equation}
where $g'$ denotes either the left or right derivative of $g$.
\end{lemma}

\begin{proof}
In the case $\gamma>1$, the proof for the two estimates \eqref{recuree1pm15} and \eqref{recuree1pm18} are the same as that of \cite[Lemma~3.1]{Bergweilereremenko2019}. To prove \eqref{recuree1pm16}, we note that \eqref{recuree1pm15} implies that $g(x)=x(1+\omega(x))$ where $\omega(x)\to 0$ as $x\to\infty$. Using \eqref{recuree1pm15} again we see that
\begin{equation*}
\begin{split}
h(g(x))+\mathfrak{h}(g(x))=x^{\gamma}=x^{\gamma}[1+\gamma\omega(x)+O(\omega(x)^2)]+O(x^{\gamma-2})+O(1)+O(\alpha(x)x^{\gamma})
\end{split}
\end{equation*}
as $x\to\infty$. Note that the two terms $O(x^{\gamma-2})$ and $O(1)$ above vanish in the case $\gamma=1$. This yields
\begin{equation*}
\begin{split}
\omega(x)=O(\alpha(x))+O(x^{-2})+O(x^{-\gamma}),
\end{split}
\end{equation*}
from which \eqref{recuree1pm16} follows. Similarly as in \cite[Lemma~3.1]{Bergweilereremenko2019}, we have
\begin{equation*}
\begin{split}
(h+\mathfrak{h})'(x)=\gamma x^{\gamma-1}+O(x^{\gamma-2})+O(1)+O(\alpha(x)x^{\gamma-1})
\end{split}
\end{equation*}
and thus
\begin{equation*}
\begin{split}
(h+\mathfrak{h})'(g(x))=\gamma x^{\gamma-1}+O(x^{\gamma-2})+O(1)+O(\alpha(x)x^{\gamma-1}).
\end{split}
\end{equation*}
Hence
\begin{equation*}
\begin{split}
g'(x)=\frac{\gamma x^{\gamma-1}}{(h+\mathfrak{h})'(g(x))}=1+O(\alpha(x))+O(x^{-1})+O(x^{1-\gamma}),
\end{split}
\end{equation*}
which is \eqref{recuree1pm17}. This completes the proof.

\end{proof}

In the proof of Theorem~\ref{maintheorem2}, the function $\mathfrak{h}(x)$ in Lemma~\ref{Lemma1} is chosen as in the following

\begin{lemma}\label{Lemma2}
Let $\delta\in[0,1]$ and $\gamma$ be the constant in Lemma~\ref{Lemma1}. Then there exists a non-negative sequence $(m_k)$ of integers, with $m_1=0$, such that the function $\mathfrak{h}:[0,\infty)\to[0,\infty)$ satisfies $\mathfrak{h}(0)=0$ and one of the following properties:

(1) when $0\leq \delta\gamma \leq 1$, $\mathfrak{h}$ is linear on each interval $[2\pi (k-1), 2\pi k]$ and has slopes $m_k$ there and there is a subsequence $(m_{k_i})$ such that $m_{k_i}=1$ for $k\geq k_1$ and $m_{k}=0$ for other $k_i$ such that
\begin{equation}\label{recuree1pm13}
\begin{split}
\mathfrak{h}(x)=
\left\{
  \begin{array}{ll}
    x^{\delta\gamma}+O(1), &  x\in[2\pi (k_i-1), 2\pi k_i] \\
    0,                          &  \text{otherwise}
  \end{array}
\right.   , \quad x\to\infty
\end{split}
\end{equation}
and, in particular, $m_k=0$ for all $k\geq k_1$ when $\delta\gamma=0$ and $m_k=1$ for all $k\geq k_1$ when $\delta\gamma=1$ and $k_{i}\geq c i^{1/\delta\gamma}$ for a fixed positive constant $c$ when $0<\delta\gamma<1$; or

(2) when $\delta\gamma>1$, $\mathfrak{h}$ is linear on the intervals $[2\pi (k-1), 2\pi k]$ and has slope $m_k$ there satisfying
\begin{equation}\label{recuree1pm13fu1}
\begin{split}
\mathfrak{h}(x)=\alpha(x)x^{\delta\gamma}+O(\alpha(x)x^{\delta\gamma-2})+O(1), \quad x\to\infty,
\end{split}
\end{equation}
and, moreover, $(m_k)$ is increasing such that $m_{k+1}-m_k$ is an odd integer when $m_{k+1}>m_k$ and
\begin{equation}\label{recuree1pm13fu2}
\begin{split}
m_k=\alpha(2\pi k)\delta\gamma(2\pi k)^{\delta\gamma-1}+O(\alpha(2\pi k)k^{\delta\gamma-2})+O(1), \quad k\to\infty.
\end{split}
\end{equation}

\end{lemma}

\begin{proof}
When $0\leq \delta\gamma \leq 1$, we set $\mathfrak{h}(0)=0$ and choose $k_0=k_1=0$. We note that
\begin{equation*}
\begin{split}
(2\pi k)^{\delta\gamma}-(2\pi(k-1))^{\delta\gamma} \leq 2\pi \quad \text{for} \quad k\geq 1.
\end{split}
\end{equation*}
For $k\leq k_{1}$ we set $m_k=0$. Suppose that $k\geq k_2$, and $\mathfrak{h}(2\pi(k_{i-1}))$ is already defined. Then we define
\begin{equation*}
\begin{split}
\mathfrak{h}(2\pi k_{i}):=2\pi i,
\end{split}
\end{equation*}
where $k_{i}$ is the largest positive integer such that
\begin{equation*}
\begin{split}
(2\pi(k_{i}-1))^{\delta\gamma}<2\pi i\leq (2\pi k_{i})^{\delta\gamma}.
\end{split}
\end{equation*}
Then we interpolate $\mathfrak{h}$ linearly between $2\pi k_{i}-1$ and $2\pi k_i$ in the way that the slopes are equal to $m_{k_i}=1$ on the interval $[2\pi (k_i-1),2\pi k_i]$ and the slopes are equal to $m_{k}=0$ otherwise. In particular, we see that $m_k=0$ for all $k\geq k_1$ when $\delta=0$ and $m_k=1$ for all $k\geq k_1$ when $\delta\gamma=1$; when $0<\delta\gamma<1$
we see that the subsequence $(m_{k_i})$ satisfies $k_{i}\geq c i^{1/\delta\gamma}$ for a fixed positive constant $c$. Thus our first assertion on $\mathfrak{h}$ in \eqref{recuree1pm13} follows.

When $\delta\gamma>1$, the proof is a slightly modification of that in \cite[Lemma~3.1]{Bergweilereremenko2019}.
We shall define $m_k$ in the following way: We set $\mathfrak{h}(0)=0$ and choose $k_0\in \mathbb{N}$ so that
\begin{equation}\label{recuree ne1}
\begin{split}
\alpha(2\pi k)(2\pi k)^{\delta\gamma}-\alpha(2\pi (k-1))(2\pi(k-1))^{\delta\gamma} \geq 4\pi \quad \text{for} \quad k\geq k_0.
\end{split}
\end{equation}
Such a $k_0$ exists when $\delta\gamma>1$. For $k\leq k_0$ we set $m_k=0$. Suppose that $k>k_0$, and $\mathfrak{h}(2\pi (k-1))$ is already defined. Then we define
\begin{equation*}
\begin{split}
\mathfrak{h}(2\pi k):=2\pi p_{k},
\end{split}
\end{equation*}
where $p_{k}$ is a positive integer of opposite parity to $\mathfrak{h}(2\pi(k-1))/2\pi$ minimizing
\begin{equation*}
\begin{split}
\left|\alpha(2\pi k)(2\pi k)^{\delta\gamma}-2\pi p_{k}\right|.
\end{split}
\end{equation*}
There are at most two such $p_k$, and when there are two, we choose the larger one. Then we interpolate $\mathfrak{h}$ linearly between $2\pi (k-1)$ and $2\pi k$. Evidently, with this definition,
\begin{equation}\label{recuree ne2}
\begin{split}
\left|\mathfrak{h}(2\pi k)-\alpha(2\pi k)(2\pi k)^{\delta\gamma}\right|\leq 2\pi.
\end{split}
\end{equation}
Using \eqref{recuree ne1} and \eqref{recuree ne2} we have
\begin{equation}\label{recuree ne2 f0}
\begin{split}
\mathfrak{h}(2\pi k)\geq \alpha(2\pi k)(2\pi k)^{\delta\gamma}-2\pi>\alpha(2\pi (k-1))(2\pi (k-1))^{\delta\gamma}+2\pi\geq \mathfrak{h}(2\pi (k-1)).
\end{split}
\end{equation}
Further, after $\mathfrak{h}(2\pi k)$ is already defined, then we define
\begin{equation*}
\begin{split}
\mathfrak{h}(2\pi (k+1)):=2\pi p_{k+1},
\end{split}
\end{equation*}
where $p_{k+1}$ is a positive integer of the \emph{same} parity to $h(2\pi k)/2\pi$ minimizing
\begin{equation*}
\begin{split}
\left|\alpha(2\pi (k+1))(2\pi (k+1))^{\delta\gamma}-2\pi p_{k+1}\right|.
\end{split}
\end{equation*}
There are at most two such $p_k$, and when there are two, we choose the larger one. Then we interpolate $\mathfrak{h}$ linearly between $2\pi k$ and $2\pi (k+1)$. Evidently, with this definition,
\begin{equation}\label{recuree ne3}
\begin{split}
\left|\mathfrak{h}(2\pi (k+1))-\alpha(2\pi (k+1))(2\pi (k+1))^{\delta\gamma}\right|\leq 2\pi.
\end{split}
\end{equation}
Using \eqref{recuree ne1} and \eqref{recuree ne3} we have
\begin{equation}\label{recuree ne2 f1}
\begin{split}
\mathfrak{h}(2\pi (k+1))\geq \alpha(2\pi (k+1))(2\pi (k+1))^{\delta\gamma}-2\pi>\alpha(2\pi k)(2\pi k)^{\delta\gamma}+2\pi\geq \mathfrak{h}(2\pi k).
\end{split}
\end{equation}
By \eqref{recuree ne2 f0} and \eqref{recuree ne2 f1} we see that $h$ is strictly increasing, and its slopes are positive integers, say $m_k$. With this definition, $m_{k+1}-m_k$ is an odd integer whenever $m_{k+1}>m_k$.

To prove that the function $h$ satisfies \eqref{recuree1pm15}, we note first that $h(2\pi k)=\alpha(2\pi k)(2\pi k)^{\delta\gamma}+O(1)$ by the construction. For $0\leq t\leq 1$ we use the relation $\log(1+x)\sim x$ as $x\to 0$ and have
\begin{equation*}
\begin{split}
\frac{\alpha(2\pi (k+t))(2\pi (k+t))^{\delta\gamma}}{\alpha(2\pi k)(2\pi k)^{\delta\gamma}}=\left(1-\frac{2t}{k\log (2\pi(k+1))}\right)\left(1+\delta\gamma \frac{t}{k}\right)+O\left(\frac{1}{k^2}\right)
\end{split}
\end{equation*}
as $k\to\infty$. Then by substituting the above relation into the following equation
\begin{equation*}
\begin{split}
\alpha(2\pi (k+t))(2\pi (k+t))^{\delta\gamma}-t\alpha(2\pi (k+1))(2\pi (k+1))^{\delta\gamma}-(1-t)\alpha(2\pi k)(2\pi k)^{\delta\gamma}
\end{split}
\end{equation*}
we conclude that the straight line connecting the points $(2\pi k, \alpha(2\pi k)(2\pi k)^{\delta\gamma})$ and $(2\pi (k+1), \alpha(2\pi (k+1))(2\pi (k+1))^{\delta\gamma})$ deviates from the graph of the function $x\to \alpha(x)x^{\delta\gamma}$ between the points $k$ and $k+1$ by a term which is $O(\alpha(x)x^{\delta\gamma-2})$. This yields \eqref{recuree1pm13fu1}. Finally, by construction we have
\begin{equation}\label{recuree ne13}
\begin{split}
m_k=\frac{\alpha(2\pi k)(2\pi k)^{\delta\gamma}-\alpha(2\pi (k-1))(2\pi (k-1))^{\delta\gamma}}{2\pi}+O(1),
\end{split}
\end{equation}
from which \eqref{recuree1pm18} follows.

\end{proof}

To prove Theorem~\ref{maintheorem2}, we also need to define the functions $h$ and $g$ in Lemma~\ref{Lemma1} in the case $\gamma=1$. In this case, we will choose $n_1=0$ and $n_k=1$ for all $k\geq 2$. Then the function $h:[0,\infty)\to[0,\infty)$ will be defined in the way that $h(x)=x$ when $k=1$ and $h(x)=3x-4\pi$ when $k\geq 2$. It follows that the function $g:[0,\infty)\to[0,\infty)$ defined by $h(g(x))=x$ satisfies $x(1+o(1))/3\leq g(x)\leq x(1+o(1))$ as $x\to\infty$.

The choice of the function $\alpha(x)$ in Lemma~\ref{Lemma2} guarantees that the function $G(z)$ constructed in subsection~\ref{Definition of a quasimeromorphic map} satisfies the hypothesis of the Teichm\"uller--Wittich--Belinskii theorem \cite[\S~V.6]{lehtoVirtanen2008} in the case $\delta=1$. If we only consider the case $\delta<1$, then $\alpha(x)$ can be replaced by the constant $1$.

Let $(n_k)$ and $(m_k)$ be the two sequences constructed in Lemma~\ref{Lemma1} and Lemma~\ref{Lemma2}, respectively. When $\gamma>1$ we have
\begin{equation}\label{recuree1pm23bea1}
\begin{split}
\frac{m_k}{2n_k+1}=O\left(\frac{1}{(\log (k+2))^{2}}\right), \quad k\to\infty
\end{split}
\end{equation}
in both of the two cases $\delta\gamma\leq 1$ and $\delta\gamma>1$. Below we shall prove Lemma~\ref{Lemma3}--Lemma~\ref{Lemma5} in the case $\gamma>1$. For simplicity, in the proofs we shall use the notation $\beta(k)$ such that $\beta(k)=O((\log (k+2))^{-2})$ where $k\to\infty$.

We denote $N_k=m_k+2n_k+1$ and $N_{k+1}=\mathfrak{m}_k+2\mathfrak{n}_k+1$. In particular, by Lemma~\ref{Lemma2} we have $N_{k+1}>N_k$ when $\delta\gamma\leq 1$ and $\gamma>1$. By \eqref{recuree1pm18} we have
\begin{equation}\label{recuree1pm71}
\begin{split}
2n_k+1=\gamma(2\pi k)^{\gamma-1}(1+O(k^{-1}))
\end{split}
\end{equation}
and hence, since $\delta\gamma\leq \gamma$, by \eqref{recuree1pm13fu2} and \eqref{recuree1pm23bea1} that
\begin{equation}\label{recuree1pm72}
\begin{split}
N_k=\gamma(2\pi k)^{\gamma-1}(1+\beta(k)).
\end{split}
\end{equation}
It follows that
\begin{equation}\label{recuree1pm73}
\begin{split}
\frac{N_{k+1}}{N_k}=\left(\frac{k+1}{k}\right)^{\gamma-1}(1+\beta(k))=1+\beta(k).
\end{split}
\end{equation}
Moreover, denoting $\mathcal{N}_k=\sum_{j=1}^{k}N_j$, we have from \eqref{recuree1pm72} that
\begin{equation}\label{recuree1pm109}
\begin{split}
\mathcal{N}_{k}\sim (2\pi)^{\gamma-1}k^{\gamma}, \quad k\to\infty.
\end{split}
\end{equation}
The above estimates \eqref{recuree1pm71}--\eqref{recuree1pm109} will be used in the following proof.

Let now $m,n,\mathfrak{m},\mathfrak{n}\in \mathbb{N}$. For each pair $(m,2n)$, we consider the function $\phi:\mathbb{R}\to\mathbb{R}$ defined by
\begin{equation}\label{diffeo}
\begin{split}
g_{\mathfrak{m},\mathfrak{n}}(x)=g_{m,n}(\phi(x)).
\end{split}
\end{equation}
Then $\phi$ is an increasing diffeomorphism on $\mathbb{R}$; see \cite{Bergweilereremenko2017} or \cite{Zhang2023}. We will consider the functions $\phi$ for the case that $m=m_k$, $n=n_k$, $\mathfrak{m}=m_{k+1}$ and $\mathfrak{n}=n_{k+1}$. Thus we will consider the behavior of $\phi$ as $k\to\infty$, but in order to simplify the formulas we suppress the dependence of $\phi$ from $m,n$ and $\mathfrak{m},\mathfrak{n}$ from the notation. We first prove the following

\begin{lemma}\label{Lemma3}
Let $m,n\in \mathbb{N}$ and put $N=m+2n+1$. Let $y>0$. Then
\begin{equation}\label{recuree1pm24}
\begin{split}
\log(h_{m,n}(y)-1)=&\ -\log \binom{m+2n}{m} N!+y+N\log y\\
&\ -2\log Q_m(y)-\log\left(1+\frac{y}{N}\right)+R(y,N),
\end{split}
\end{equation}
where $R(y,N)\leq 4Bmy/[N(y+N-2m)]$ for all $y$ and all $N\geq N_0$ for some integer $N_0$ and some constant $B$.
\end{lemma}

We see that the error term $R(y,N)$ in \eqref{recuree1pm24} is uniformly bounded for all $y>0$ and all $N\geq N_0$. We also see that, in the case $\gamma>1$, $R(y,N)$ is of type $\beta(k)$ when $N\to\infty$, instead of type $O(1/N)$ as in \cite[Lemma~3.2]{Bergweilereremenko2019}, which is enough for our purpose. In terms of $g_{m,n}$ equation \eqref{recuree1pm24} takes the form
\begin{equation}\label{recuree1pm25}
\begin{split}
\log(g_{m,n}(x)-1)=&\ -\log \binom{m+2n}{m} N!+e^x+Nx\\
&\ -2\log Q_m(e^x)-\log\left(1+\frac{e^x}{N}\right)+R(e^x,N).
\end{split}
\end{equation}

\begin{proof}[Proof of Lemma~\ref{Lemma3}]
It is easy to see from \eqref{recuree1pm9} and \eqref{recuree1pm10} that
\begin{equation*}
\begin{split}
h'_{m,n}(y)=\frac{1}{\binom{m+2n}{m}(m+2n)!}\frac{e^yy^{m+2n}}{Q_m(y)^2}.
\end{split}
\end{equation*}
Since $h_{m,n}(y)$ is analytic on $[0,+\infty)$, we have by \eqref{recuree1pm10} and the formula for the error term of a Taylor series that
\begin{equation}\label{recuree1pm27}
\begin{split}
h_{m,n}(y)-1=&\ \frac{1}{\binom{m+2n}{m}(m+2n)!}\int_0^{y}\frac{e^{u}u^{m+2n}}{Q_m(u)^2}du\\
=&\ \frac{1}{\binom{m+2n}{m}(m+2n)!}\int_0^{y}\frac{e^{u}}{Q_m(u)^2}d\left(\frac{u^{m+2n+1}}{m+2n+1}\right)\\
=&\ \frac{1}{\binom{m+2n}{m}(m+2n+1)!}\frac{e^yy^{m+2n+1}}{Q_m(y)^2}\left(1-U(y)\right),
\end{split}
\end{equation}
where
\begin{equation}\label{recuree1pm27 fur1}
\begin{split}
U(y)=\frac{Q_m(y)^2}{e^yy^{m+2n+1}}\int_0^{y}\frac{e^{u}u^{m+2n+1}[Q_m(u)-2Q'_m(u)]}{Q_m(u)^3}du.
\end{split}
\end{equation}
Note that $N=m+2n+1$. We write
\begin{equation}\label{recuree1pm27 fur2}
\begin{split}
U(y)=\frac{Q_m(y)^2}{e^yy^{N}}\int_0^{y}\frac{e^{u}u^{N}}{Q_m(u)^2}du-2\frac{Q_m(y)^2}{e^yy^{N}}\int_0^{y}\frac{e^{u}u^{N}Q'_m(u)}{Q_m(u)^3}du=G_1-2G_2.
\end{split}
\end{equation}
For $G_1$ in \eqref{recuree1pm27 fur2}, we write
\begin{equation}\label{recuree1pm27 fur2 rqr1}
\begin{split}
G_1=y\int_0^{1}e^{y(s-1)}s^{N}\frac{Q_m(y)^2}{Q_m(sy)^2}du=yI.
\end{split}
\end{equation}
Since $Q_m(y)^2/Q_m(sy)^2\geq 1$ for all $s\in[0,1]$, we may first use the inequality $\log s\leq s-1$ for $s>0$ to obtain
\begin{equation}\label{recurfai   1}
\begin{split}
I\geq \int_0^{1}y^{\log s}s^{N}du=\int_0^{1}s^{y+N}du=\frac{1}{y+N+1}.
\end{split}
\end{equation}
On the other hand, since $0\leq s\leq 1$, we have $Q_m(y)^2/Q_m(ys)^2\leq s^{-2m}$ for all $s\in[0,1]$. With $\mu=\min\{1/2,1/\sqrt{y}\}$, we write
\begin{equation}\label{recurfai   2}
\begin{split}
I-\frac{1}{y+N-2m+1}=&\ \int_0^{1}\left(e^{y(s-1)}s^{N}\frac{Q_m(y)^2}{Q_m(ys)^2}-e^{y\log s}s^{N-2m}\right)ds\\
\leq&\ \int_0^{1}\left(e^{y(s-1)}s^{N}s^{-2m}-e^{y\log s}s^{N-2m}\right)ds= I_1+I_2,
\end{split}
\end{equation}
where
\begin{equation}\label{recurfai   3}
\begin{split}
I_1=&\ \int_0^{\mu}\left(e^{y(s-1)}s^{N}s^{-2m}-e^{y\log s}s^{N-2m}\right)ds,\\
I_2=&\ \int_{\mu}^{1}\left(e^{y(s-1)}s^{N}s^{-2m}-e^{y\log s}s^{N-2m}\right)ds.
\end{split}
\end{equation}
Then by the same arguments as that in the proof of \cite[Lemma~3.2]{Bergweilereremenko2019}, we have
\begin{equation}\label{recurfai   4}
\begin{split}
I_1\leq \int_0^{\mu}e^{y(s-1)}s^{N}s^{-2m}\leq \frac{8}{(y+N-2m)^2}
\end{split}
\end{equation}
and
\begin{equation}\label{recurfai   5}
\begin{split}
I_2\leq \int_0^{\mu}e^{y(s-1)}s^{N}s^{-2m}\leq \frac{4}{(y+N-2m)^2}.
\end{split}
\end{equation}
From \eqref{recurfai   1}--\eqref{recurfai   5} we get
\begin{equation}\label{recurfai   5 fur1}
\begin{split}
\frac{1}{y+N+1}\leq I \leq \frac{1}{y+N-2m+1}+\frac{12}{(y+N-2m)^2}.
\end{split}
\end{equation}
Since
\begin{equation*}
\begin{split}
\frac{1}{y+N+1}-\frac{1}{y+N}=-\frac{1}{(y+N+1)(y+N)}
\end{split}
\end{equation*}
and
\begin{equation*}
\begin{split}
\frac{1}{y+N-2m+1}-\frac{1}{y+N}=\frac{2m-1}{(y+N-2m+1)(y+N)},
\end{split}
\end{equation*}
we have
\begin{equation}\label{recurfai   5 fur4}
\begin{split}
\left|I-\frac{1}{y+N}\right|\leq \frac{2m+11}{(y+N-2m)^2}.
\end{split}
\end{equation}
Now, for $G_2$ in \eqref{recuree1pm27 fur2}, we recall that $Q_m(y)=\sum_{i=0}^{m}A_ie^{iy}$ and $A_i=\frac{(2n)!}{(i+2n)!}$. Then, since $i\leq i+2n-1$ for all $i$, we have, for all $y>0$,
\begin{equation*}
\begin{split}
\frac{Q'_m(y)}{Q_m(y)}=\frac{\sum_{i=1}^{m}iA_ie^{iy}}{\sum_{i=0}^{m}A_ie^{iy}}\leq \frac{m}{1+2n}.
\end{split}
\end{equation*}
Thus we have
\begin{equation}\label{recurfai   8}
\begin{split}
G_2=\frac{Q_m(y)^2}{e^yy^{N}}\int_0^{y}\frac{e^{u}u^{N}Q'_m(u)}{Q_m(u)^3}du\leq \frac{m}{1+2n}G_1,
\end{split}
\end{equation}
By combining \eqref{recuree1pm27 fur1}, \eqref{recurfai   5 fur4} and \eqref{recurfai   8}, we have
\begin{equation*}
\begin{split}
\log(1-U(y))=&\ \log \left(1-G_1+2G_2\right)\\
=&\ \log \left(\frac{N}{y+N}-r(y,N)+2G_2\right)\\
=&\ -\log \left(1+\frac{y}{N}\right)+\log \left(1-\frac{y+N}{N}\left(r(y,N)-2G_2\right)\right),
\end{split}
\end{equation*}
where $r(y,N)$ satisfies
\begin{equation}\label{recurfai   10}
\begin{split}
|r(y,N)|\leq \frac{(2m+11)y}{(y+N-2m)^2}.
\end{split}
\end{equation}
Note that $2n=N-m-1$. Together with \eqref{recuree1pm27 fur2 rqr1}, \eqref{recurfai   5 fur1}, \eqref{recurfai   8} and \eqref{recurfai   10}, we have
\begin{equation*}
\begin{split}
&\ \frac{y+N}{N}(|r(y,N)|+2G_2)\\
\leq&\ \frac{y+N}{N}\left[\left(2m+11+\frac{24m}{1+2n}\right)\frac{y}{(y+N-2m)^2}+\frac{2m}{1+2n}\frac{y}{y+N-2m}\right]\\
\leq&\ \frac{y+N}{N}\left(\frac{(2m+35)y}{(y+N-2m)^2}+\frac{2m}{1+2n}\frac{y}{y+N-2m}\right).
\end{split}
\end{equation*}
By the constructions of $(n_k)$ and $(m_k)$ in Lemma~\ref{Lemma1} and Lemma~\ref{Lemma2}, we see that there is an integer $N_0$ and some constant $B$ such that for all $N\geq N_0$
\begin{equation}\label{recurfai   12}
\begin{split}
\frac{y+N}{N}(r(y,N)+2G_2) \leq \frac{4Bmy}{N(y+N-2m)}\leq \frac{1}{2}.
\end{split}
\end{equation}
Therefore, using the inequality that $|\log (1+t)|\leq 2|t|$ for $|t|\leq 1/2$, it follows that
\begin{equation}\label{recurfai   13}
\begin{split}
\log(1-U(y))\leq \frac{4Bmy}{N(y+N-2m)}
\end{split}
\end{equation}
for all $N\geq N_0$. Then, by taking the logarithm on both sides of equation \eqref{recuree1pm27} together with \eqref{recurfai   13}, we have the estimate in \eqref{recuree1pm24}. Thus the lemma is proved.

\end{proof}

Recall that by \eqref{recuree1pm10} the function $g_{m,n}:\mathbb{R}\to(1,\infty)$ is an increasing homeomorphism. Thus there exists $s_{m,n}\in \mathbb{R}$ such that
\begin{equation}\label{trans 0}
\begin{split}
g_{m,n}(s_{m,n})=2.
\end{split}
\end{equation}
For each pair of $(m,2n)$, as in \cite{Bergweilereremenko2019} we shall work with the function $g_{m,n}(z+s_{m,n})$ instead of the function $g_{m,n}(z)$ directly. We prove the following

\begin{lemma}\label{Lemma3 trans}
The constant $s_{m,n}$ in \eqref{trans 0} satisfies $s_{m,n}=\log N+O(1)$ as $k\to\infty$.
\end{lemma}

\begin{proof}
We write $s_{m,n}=\log N+r$. Then equation \eqref{recuree1pm25} yields
\begin{equation*}
\begin{split}
0=&\ -\log \binom{m+2n}{m}-\log N!+Ne^r+N\log N+Nr\\
&\ -2\log Q_m(Ne^r)-\log(1+e^r)+\beta(k).
\end{split}
\end{equation*}
By Stirling's formula we have
\begin{equation}\label{trans 2}
\begin{split}
N(e^r+r+1)-\log(1+e^r)-2\log Q_m(Ne^r)=\log \binom{m+2n}{m}+\frac{1}{2}\log N+O(1).
\end{split}
\end{equation}
When $m=0$, then from \cite[Lemma~3.3]{Bergweilereremenko2019} we know that $r=O(1)$. So we suppose that $m\geq 1$. Note that
\begin{equation}\label{trans 3}
\begin{split}
0<\log \binom{m+2n}{m}=\log\frac{(m+2n)!}{m!(2n)!}\leq m\log N.
\end{split}
\end{equation}
By \eqref{recuree1pm72} we see that $\log N=O(\log k)$. When $r\leq 0$, if $r$ has large modulus, then by the estimates in \eqref{recuree1pm23bea1} and \eqref{recuree1pm72} we have $\frac{N}{i+2n}e^{r}<1/2$ for all $i=1,\cdots,m$ so that $\log Q_m(Ne^r)=O(1)$. Thus by \eqref{trans 2} we see that $r=O(1)$ when $r\leq 0$. On the other hand, if $r>0$, we first suppose that $r$ is large. If $\delta\gamma\leq 1$, then $m=1$ and thus $\log Q_1(Ne^r)=\log(1+e^r)=r+O(1)$; if $\delta\gamma>1$, recalling that $A_i=\frac{(2n)!}{(i+2n)!}$, $i=1,\cdots,m$, then in a similar way we have
\begin{equation}\label{trans 4}
\begin{split}
\log Q_m(Ne^r)=mr+\log (N^{m}A_{m})+O(1)=mr+O(1).
\end{split}
\end{equation}
Together with \eqref{trans 3} and \eqref{trans 4} and also the estimates in \eqref{recuree1pm23bea1} and \eqref{recuree1pm72}, equation \eqref{trans 2} implies that
\begin{equation}\label{trans 5}
\begin{split}
e^r+r+1-\frac{2m+1}{N}r=O(1).
\end{split}
\end{equation}
Again, by the estimates in \eqref{recuree1pm23bea1} and \eqref{recuree1pm72}, we see from \eqref{trans 5} that $r\to\infty$ is impossible. Therefore, we must also have $s_{m,n}=\log N+O(1)$.

\end{proof}

Let $M,N$ be two positive integers such that $M>N$. We put $N=m+2n+1$ and $M=\mathfrak{m}+2\mathfrak{n}+1$. From now on we shall assume that there exists a constant $C>1$ such that $\mathfrak{m}+2\mathfrak{n}\leq C(m+2n)$. Then clearly
\begin{equation}\label{recuree1pm31}
\begin{split}
M\leq CN.
\end{split}
\end{equation}
The estimate in \eqref{recuree1pm73} implies that $C\to1$ as $N\to\infty$ by \eqref{recuree1pm18}. Below we begin to analyse the asymptotic behavior of $\phi(z)$ and prove the following

\begin{lemma}\label{Lemma4}
Let $M,N$ be two integers as before. Then
\begin{itemize}
  \item [(1)] if $\mathfrak{m}>m$, then $\phi(x)$ has at least one fixed point $p$ and each fixed point $p$ of $\phi$ satisfies $p=\log N+O(1)$ as $k\to\infty$;
  \item [(2)] if $\mathfrak{m}=m$, then $\phi(x)$ has uniquely one fixed point $p$ such that $p=\log N+O(1)$ as $k\to\infty$;
  \item [(3)] if $\mathfrak{m}<m$, then $\phi(x)$ has uniquely one fixed point $p$ such that $p=\frac{l-1}{l}\log N+O(1)$ for an integer $l=M-N$ as $k\to\infty$;
  \item [(4)] in all of the above three cases, $\phi(x)<x$ for $x<p+O(1)$ and $\phi(x)>x$ for $x>p+O(1)$ and, moreover, we have $\phi(x)\leq \log N+O(1)$ for all $x\leq\log N$.
\end{itemize}

\end{lemma}

\begin{proof}
Let $d:\mathbb{R}\to\mathbb{R}$, $d(x)=g_{\mathfrak{m},\mathfrak{n}}(x)-g_{m,n}(x)$. We consider the zeros of $d$. Now
\begin{equation}\label{recuree1pm32fu1}
\begin{split}
d'(x)=\left[e^{(M-N)x}-\frac{\binom{\mathfrak{m}+2\mathfrak{n}}{\mathfrak{m}} (\mathfrak{m}+2\mathfrak{n})!}{\binom{m+2n}{m} (m+2n)!}\frac{\left(\sum_{\mathfrak{i}=0}^{\mathfrak{m}}A_\mathfrak{i}e^{\mathfrak{i}x}\right)^2}{\left(\sum_{i=0}^{m}A_ie^{ix}\right)^2}\right]D(e^x),
\end{split}
\end{equation}
where
\begin{equation*}
\begin{split}
D(e^x)=\frac{1}{\binom{\mathfrak{m}+2\mathfrak{n}}{\mathfrak{m}} (\mathfrak{m}+2\mathfrak{n})!}\frac{e^{(m+2n+1)x}e^{e^{x}}}{\left(\sum_{i=0}^{\mathfrak{m}}A_ie^{ix}\right)^2}.
\end{split}
\end{equation*}
Note that $D(e^x)\not=0$ for all $x\in \mathbb{R}$. By Lemma~\ref{Lemma2} it is possible that $\mathfrak{m}<m$ or $\mathfrak{m}=m$ or $\mathfrak{m}>m$, and the case $\mathfrak{m}<m$, i.e., $\mathfrak{m}=0$ and $m=1$, occurs only when $\delta\gamma<1$.

When $M>N$, from \eqref{recuree1pm25} we see that $d(x)<0$ if $x$ is negative and of sufficiently large modulus and increases to $+\infty$ as $x\to+\infty$. This implies that $\phi$ has at least one fixed point.
In particular, if $\mathfrak{m}\leq m$, since $e^{(M-N)x}$ is strictly increasing on $\mathbb{R}$, \eqref{recuree1pm32fu1} implies that the derivative $d'(x)$ has exactly one sign change, namely at the point $q$, such that $d'(x)<0$ when $x<q$ and $d'(x)>0$ when $x>q$. This implies that $\phi$ has exactly one fixed point $p>q$. In the case $\mathfrak{m}>m$, it is possible that $(M-N)\geq 2(\mathfrak{m}-m)$ or $(M-N)<2(\mathfrak{m}-m)$, the derivative $d'(x)$ may have more than one sign change, namely at the point $q$, which implies that $\phi$ may have more than one fixed point.

We denote $p$ to be a fixed point of $\phi$. To determine the asymptotic behavior of $p$ as $N\to\infty$, we
note \eqref{recuree1pm25} implies
\begin{equation}\label{recuree1pm32fu2}
\begin{split}
&\ -\log \binom{\mathfrak{m}+2\mathfrak{n}}{\mathfrak{m}}M!+Mp-2\log Q_\mathfrak{m}(e^p)-\log\left(1+\frac{e^p}{M}\right)\\
=&\ -\log \binom{m+2n}{m}N!+Np-2\log Q_m(e^p)-\log\left(1+\frac{e^p}{N}\right)+\beta(k).
\end{split}
\end{equation}
Firstly, in the case $\mathfrak{m}>m$, we write $p=\log N+r$. Then it follows from \eqref{recuree1pm32fu2} that
\begin{equation*}
\begin{split}
(M-N)r=\log \frac{\binom{\mathfrak{m}+2\mathfrak{n}}{\mathfrak{m}}M!}{\binom{m+2n}{m}N!}-(M-N)\log N+\log\frac{1+\frac{N}{M}e^r}{1+e^r}+2\log\frac{Q_\mathfrak{m}(Ne^r)}{Q_m(Ne^r)}+\beta(k).
\end{split}
\end{equation*}
Now,
\begin{equation*}
\begin{split}
&\ \left|\log\frac{1+\frac{N}{M}e^r}{1+e^r}\right|=\left|\log\left(1-\frac{M-N}{M}\frac{e^r}{1+e^r}\right)\right|\\
\leq &\ \log\left(1-\frac{M-N}{M}\right) \leq \left|\log\frac{N}{M}\right|=\log\left(1+\frac{M-N}{N}\right)\leq \frac{M-N}{N}.
\end{split}
\end{equation*}
Hence
\begin{equation*}
\begin{split}
r=\frac{\log \frac{\binom{\mathfrak{m}+2\mathfrak{n}}{\mathfrak{m}}}{\binom{m+2n}{m}}+
\log \frac{M!}{N!}-(M-N)\log N+2\log\frac{Q_\mathfrak{m}(Ne^r)}{Q_m(Ne^r)}}{M-N} + \beta(k).
\end{split}
\end{equation*}
An application of Stirling's formula now yields
\begin{equation*}
\begin{split}
r=\frac{\log \frac{\binom{\mathfrak{m}+2\mathfrak{n}}{\mathfrak{m}}}{\binom{m+2n}{m}}}{M-N}
+\frac{M\log\frac{M}{N}}{M-N}-1-
\frac{1}{2}\frac{\log\frac{M}{N}}{M-N}+2\frac{\log\frac{Q_\mathfrak{m}(Ne^r)}{Q_m(Ne^r)}}{M-N}+\beta(k).
\end{split}
\end{equation*}
Note that
\begin{equation}\label{recuree1pm32fu2 fq5}
\begin{split}
\frac{\log \frac{M}{N}}{M-N}\sim\frac{\log \frac{M}{N}}{N(M/N-1)}\leq \frac{1}{N}.
\end{split}
\end{equation}
Thus we may write the asymptotic formula for $r$ in the form
\begin{equation}\label{recuree1pm33}
\begin{split}
r=\frac{\log\frac{\binom{\mathfrak{m}+2\mathfrak{n}}{\mathfrak{m}}}{\binom{m+2n}{m}}}{M-N}+
2\frac{\log\frac{Q_\mathfrak{m}(Ne^r)}{Q_m(Ne^r)}}{M-N}+\beta(k).
\end{split}
\end{equation}
We need to show that $r=O(1)$ as $N\to\infty$. When $0\leq \delta\gamma\leq 1$, we know from Lemma~\ref{Lemma2} that $m_k$ is equal to either $0$ or $1$. Since $\mathfrak{m}>m$, we have
\begin{equation*}
\begin{split}
1<\frac{(\mathfrak{m}+2\mathfrak{n})!}{(m+2n)!}\frac{m!}{\mathfrak{m}!}\frac{(2n)!}{(2\mathfrak{n})!}=\frac{(\mathfrak{m}+2\mathfrak{n})\cdots(m+2n+1)}{\mathfrak{m}\cdots (m+1)\cdot(2\mathfrak{n})\cdots(2n+1)}\leq C^{M-N}.
\end{split}
\end{equation*}
Therefore,
\begin{equation}\label{recuree1pm33fu3}
\begin{split}
0<\log \frac{\binom{\mathfrak{m}+2\mathfrak{n}}{\mathfrak{m}}}{\binom{m+2n}{m}}=\log \frac{[(\mathfrak{m}+2\mathfrak{n})!][m!][(2n)!]}{[(m+2n)!][\mathfrak{m}!][(2\mathfrak{n})!]}\leq (M-N)\log C.
\end{split}
\end{equation}
Moreover, by \eqref{recuree1pm23bea1} and the function $x\to (x\log x)/(x-1)$ is increasing on the interval $(1,C]$ and since $\lim_{x\to1}(x\log x)/(x-1)=1$, we deduce from \eqref{recuree1pm33} that
\begin{equation}\label{recuree1pm33fu4h0}
\begin{split}
r=2\frac{\log\frac{Q_\mathfrak{m}(Ne^r)}{Q_m(Ne^r)}}{M-N}+O(1).
\end{split}
\end{equation}
Recall that $A_i=\frac{(2n)!}{(i+2n)!}$, $i=1,\cdots,m$. When $r\leq 0$, if $r$ has large modulus, then by \eqref{recuree1pm71} and \eqref{recuree1pm72} we see that $\frac{N}{i+2n}$ is bounded and thus $\frac{N}{i+2n}e^{r}<1/2$ for all $i=1,\cdots,m$ so that $\log Q_\mathfrak{m}(Ne^r)/Q_m(Ne^r)=O(1)$. Thus by \eqref{recuree1pm33} we see that $r=O(1)$ when $r\leq 0$. On the other hand, if $r>0$, we suppose that $r$ is large. If $\delta\gamma\leq 1$, then $\mathfrak{m}=1$ and $m=0$ and thus $\log Q_1(Ne^r)=\log(1+e^r)=r+o(1)$. Since $M-N=2(\mathfrak{n}-n)-1$ we see that $M-N=1,3,\cdots$, which together with \eqref{recuree1pm33fu4h0} implies that $r=O(1)$. If $\delta\gamma>1$, in a similar way we have
\begin{equation}\label{recuree1pm33fu4h0ar1}
\begin{split}
\log Q_\mathfrak{m}(Ne^r)=\mathfrak{m}r+\log (N^{\mathfrak{m}}A_{\mathfrak{m}})+O(1)=\mathfrak{m}r+O(1).
\end{split}
\end{equation}
Moreover, by \eqref{recuree1pm13fu2} we see that $\mathfrak{m}\leq 2m$. Together with \eqref{recuree1pm32fu2 fq5} we have
\begin{equation}\label{recuree1pm33fu4h0ar1ogu}
\begin{split}
\frac{\log\frac{Q_\mathfrak{m}(Ne^r)}{Q_m(Ne^r)}}{M-N}=\frac{\mathfrak{m}-m}{M-N}r+ \frac{(\mathfrak{m}-m)\log\frac{N}{M}}{M-N}+O(1)=\frac{\mathfrak{m}-m}{M-N}r+O(1).
\end{split}
\end{equation}
By Lemma~\ref{Lemma2}, $\mathfrak{m}-m$ is an odd integer. Since $M-N=2(\mathfrak{n}-n)-(\mathfrak{m}-m)$, we see that $2(\mathfrak{m}-m)/(M-N)\not=1$. Recall that $m_k$ is obtained in a similar way as $2n_k+1$ obtained from \eqref{recuree ne13}. Then we have
\begin{equation}\label{recuree1pm18uiyjpi}
\begin{split}
\mathfrak{m}-m=m_{k+1}-m_k=\delta\gamma(\delta\gamma-1)\alpha(2\pi k)(2\pi)^{\delta\gamma-1}k^{\delta\gamma-2}+O(k^{\delta\gamma-3})+O(1),
\end{split}
\end{equation}
as $k\to\infty$. If $1<\delta\gamma\leq 2$, then $\mathfrak{m}-m$ is bounded and thus we may obtain that $r=O(1)$ just as in the case $\mathfrak{m}=1$ and $m=0$. If $\delta\gamma>2$, then we also have $\gamma\geq \delta\gamma>2$. It follows from the estimate in \eqref{recuree ne13} that
\begin{equation}\label{recuree1pm72ij;f}
\begin{split}
2(\mathfrak{n}-n)=2(n_{k+1}-n_k)=\gamma(\gamma-1)(2\pi)^{\gamma-1}k^{\gamma-2}(1+\beta(k)).
\end{split}
\end{equation}
When $k$ is large, then from \eqref{recuree1pm18uiyjpi} and \eqref{recuree1pm72ij;f} we see that $\frac{\mathfrak{m}-m}{M-N}\leq \varepsilon$ for a small $\varepsilon>0$. Then we see from \eqref{recuree1pm33fu4h0} and \eqref{recuree1pm33fu4h0ar1ogu} that $r=O(1)$. From the above discussions, we also have $r=O(1)$ when $r>0$. We conclude that $p=\log N+O(1)$ when $\mathfrak{m}>m$. Thus we have the first assertion.

Secondly, in the case $\mathfrak{m}=m$, by Lemma~\ref{Lemma2} we must have $\delta\gamma\leq 1$ and $\mathfrak{m}=m=0$ or $\mathfrak{m}=m=1$. When $\mathfrak{m}=m=0$, we know from \cite{Bergweilereremenko2019} that $r=o(1)$ as $N\to\infty$. When $\mathfrak{m}=m=1$, we also write $p=\log N+r$. Now, the two terms $2\log Q_\mathfrak{m}(e^p)$ and $2\log Q_m(e^p)$ in \eqref{recuree1pm32fu2} cancel out. Then by following the same arguments as in previous case we easily obtain that $r=\beta(k)$. Thus we have second assertion.

Thirdly, in the case $\mathfrak{m}<m$, by Lemma~\ref{Lemma2} we must have $\delta\gamma<1$ and $\mathfrak{m}=0$ and $m=1$. Suppose that $M-N=l$ for an integer $l\geq 1$ and write $p=\frac{l-1}{l}\log N+r$. Since $\binom{\mathfrak{m}+2\mathfrak{n}}{\mathfrak{m}}/\binom{m+2n}{m}=1/N$ and $M-N=l$, then by similar arguments as before, we obtain, instead of equation \eqref{recuree1pm33fu4h0}, that
\begin{equation}\label{recuree1pm32fu2 fq4iul}
\begin{split}
r=2\frac{-\log Q_1(N^{(l-1)/l}e^r)}{M-N}+\beta(k).
\end{split}
\end{equation}
Note that $(l-1)/l\leq 1$. When $r\leq 0$, we only need to follow the same arguments as in previous case to obtain that $r=O(1)$; when $r\geq 0$, since $Q_1(N^{(l-1)/l}e^r)>1$, we see from \eqref{recuree1pm32fu2 fq4iul} that $r=O(1)$. Thus we have the third assertion.

Finally, together with the zeros of $d(x)$ we may say that $\phi(x)\leq x$ when $x<p+O(1)$ and also that $\phi(x)\geq x$ when $x>p+O(1)$ in all the three cases $\mathfrak{m}>m$, $\mathfrak{m}=m$ and $\mathfrak{m}<m$. If $\mathfrak{m}\geq m$, then from the first and second assertions we have $\phi<x$ when $x<p=\log N+O(1)$ and thus $\phi<x+O(1)\leq \log N+O(1)$ when $x\leq \log N$. If $\mathfrak{m}<m$, then $\mathfrak{m}=0$ and $m=1$ and we have from \eqref{recuree1pm25} that
\begin{equation}\label{recuree1pm32fu2guyki}
\begin{split}
&\ -\log \binom{\mathfrak{m}+2\mathfrak{n}}{\mathfrak{m}}M!+e^{\phi}+M\phi-2\log Q_\mathfrak{m}(e^{\phi})-\log\left(1+\frac{e^{\phi}}{M}\right)\\
=&\ -\log \binom{m+2n}{m}N!+e^{x}+Nx-2\log Q_m(e^x)-\log\left(1+\frac{e^x}{N}\right)+\beta(k).
\end{split}
\end{equation}
Suppose that $\phi-\log N\geq 0$ is unbounded. Note that $\log Q_\mathfrak{m}(e^x)=0$ and $\log Q_\mathfrak{m}(e^x)=O(\log N)$. Also, $\frac{e^x}{N}\leq 1$ and
\begin{equation*}
\begin{split}
\log\left(1+\frac{N}{M}e^{\phi-\log N}\right)=\log\left(\frac{N}{M}e^{\phi-\log N}\right)+O(1).
\end{split}
\end{equation*}
Note that $\log N=O(\log k)$. By the estimate in \eqref{recuree1pm73} we have $(\log N)/N)=\beta(k)$. Then, together with Stirling's formula, it follows from \eqref{recuree1pm32fu2guyki} that
\begin{equation}\label{recuree1pm32fu2guyki1}
\begin{split}
e^{\phi-\log N}+\frac{M-1}{N}(\phi-\log N)=\frac{e^x}{N}+x-\log N+O(1).
\end{split}
\end{equation}
However, since $x\leq \log N$, if we let $\phi-\log N\to\infty$, then we get a contradiction from \eqref{recuree1pm32fu2guyki1}. This implies that we always have $\phi\leq \log N+O(1)$ when $x\leq \log N$. Thus we have the fourth assertion and also complete the proof.

\end{proof}

We note that $\mathfrak{m}\geq m$ for all $k\ge1$ when we assume that $\delta\gamma\geq 1$ in Lemma~\ref{Lemma2}. We also note that, if $\gamma>2$, writing $\epsilon=\min\{1,\gamma-2-\varepsilon\}$ for a small $\varepsilon>0$, then by \eqref{recuree1pm18uiyjpi} and \eqref{recuree1pm72ij;f} we have
\begin{equation*}
\begin{split}
\frac{\log N_k}{N_{k+1}-N_k}=O\left(\frac{1}{k^{\epsilon}}\right)=\beta(k).
\end{split}
\end{equation*}
This implies that the first term in \eqref{recuree1pm33} is always of type $\beta(k)$ in all of the three cases $\mathfrak{m}>m$, $\mathfrak{m}=m$ and $\mathfrak{m}<m$. Then from the above discussions in the proof of Lemma~\ref{Lemma4} we always have $p=\log N+O(1)$.

In the proof of Lemma~\ref{Lemma3 trans} and Lemma~\ref{Lemma4}, though the error terms $\beta(k)=O((\log(k+2))^{-2})$ and $O(1)$ may depend on $r$, we may suitably choose a new $O(1)$ so that they are only dependent on the constant $C$ in \eqref{recuree1pm31}. Below we shall always assume that the terms $O(1)$ appearing are only dependent on the constant $C$, not on other variables.

\begin{lemma}\label{Lemma5}
For $\mathfrak{m},\mathfrak{n},m,n,M,N$ and $\phi:\mathbb{R}\to\mathbb{R}$ there exist positive constants $c_1,\cdots,c_8$ depending only on the constant $C$ in \eqref{recuree1pm31} such that
\begin{equation}\label{recuree1pm36}
\begin{split}
|\phi(x)-x|\leq c_1e^{-x/2}\leq \frac{c_1}{N} \quad \text{for} \quad x>8\log N
\end{split}
\end{equation}
and
\begin{equation}\label{recuree1pm37}
\begin{split}
\left|\phi(x)-\frac{M}{N}x+\frac{1}{N}\log\frac{M!}{N!}\right|\leq c_2e^x \quad \text{for} \quad x<\log N.
\end{split}
\end{equation}
Moreover, with $s_{m,n}$ defined in \eqref{trans 0} we have
\begin{equation}\label{recuree1pm38}
\begin{split}
\left|\phi(x)-x\right|\leq c_3 \quad \text{for} \quad x>s_{m,n}
\end{split}
\end{equation}
and
\begin{equation}\label{recuree1pm39}
\begin{split}
\left|\phi(x)-\frac{M}{N}x+\frac{1}{N}\log\frac{M!}{N!}\right|\leq c_4 \quad \text{for} \quad x<\log N.
\end{split}
\end{equation}
Finally,
\begin{equation}\label{recuree1pm40}
\begin{split}
|\phi'(x)-1|\leq c_5\left(e^{-x/2}+\frac{1}{(\log(k+2))^2}\right)\leq \frac{c_5}{(\log(k+2))^2} \quad \text{for} \quad x>8\log N
\end{split}
\end{equation}
and
\begin{equation}\label{recuree1pm41}
\begin{split}
\left|\phi'(x)-\frac{M}{N}\right|\leq \frac{c_6e^x}{N} \quad \text{for} \quad x<\log N,
\end{split}
\end{equation}
as well as
\begin{equation}\label{recuree1pm42}
\begin{split}
c_7\leq |\phi'(x)|\leq c_8 \quad \text{for all} \quad x\in \mathbb{R}.
\end{split}
\end{equation}

\end{lemma}

\begin{proof}
Let $y=\phi(x)$ so that $g_{\mathfrak{m},\mathfrak{n}}(x)=g_{m,n}(y)$. Recall that in \eqref{recuree1pm24} the term $R(y,N)=\beta(k)$. Therefore, by \eqref{recuree1pm25} we have
\begin{equation}\label{recuree1pm42fu1}
\begin{split}
&\ -\log \binom{\mathfrak{m}+2\mathfrak{n}}{\mathfrak{m}} M!+e^x+Mx-2\log Q_\mathfrak{m}(e^x)-\log\left(1+\frac{e^x}{M}\right)\\
=&\ -\log \binom{m+2n}{m} N!+e^y+Ny-2\log Q_m(e^y)-\log\left(1+\frac{e^y}{N}\right)+\beta(k).
\end{split}
\end{equation}
Suppose first $x<\log N$. Then by Lemma~\ref{Lemma4} we have $y<\log N+O(1)$. It follows that
\begin{equation*}
\begin{split}
y-\frac{M}{N}x+\frac{1}{N}\log \frac{N!}{M!}=&\ \frac{1}{N}\left(\log\frac{\binom{m+2n}{m}}{\binom{\mathfrak{m}+2\mathfrak{n}}{\mathfrak{m}}}+e^x-e^y-2\log\frac{Q_\mathfrak{m}(e^x)}{Q_m(e^y)}\right)\\
&\ -\frac{1}{N}\left(\log\left(1+\frac{e^x}{M}\right)+\log\left(1+\frac{e^y}{N}\right)+\beta(k)\right).
\end{split}
\end{equation*}
Thus
\begin{equation*}
\begin{split}
\left|y-\frac{M}{N}x+\frac{1}{N}\log \frac{N!}{M!}\right|\leq&\ \frac{1}{N}\left(\left|\log \frac{\binom{\mathfrak{m}+2\mathfrak{n}}{\mathfrak{m}}}{\binom{m+2n}{m}}\right|+e^x+e^y+2\left|\log\frac{Q_m(e^y)}{Q_\mathfrak{m}(e^x)}\right|\right)\\
&\ +\frac{1}{N}\left(\frac{e^x}{M}+\frac{e^y}{N}+\beta(k)\right).
\end{split}
\end{equation*}
Recall that $\binom{\mathfrak{m}+2\mathfrak{n}}{\mathfrak{m}}/\binom{m+2n}{m}=-\log N$ when $\mathfrak{m}<m$ and also the estimate in \eqref{recuree1pm33fu4h0ar1} when $\mathfrak{m}\geq m$.
Then, together with the estimate in \eqref{recuree1pm33fu3}, we have
\begin{equation*}
\begin{split}
\left|y-\frac{M}{N}x+\frac{1}{N}\log \frac{N!}{M!}\right|\leq c+\frac{d(e^x+e^y)}{N}.
\end{split}
\end{equation*}
This yields \eqref{recuree1pm37} and also \eqref{recuree1pm39} by Lemma~\ref{Lemma4}.

Suppose now that $x>\log N$. Using
\begin{equation*}
\begin{split}
\log\left(1+\frac{e^x}{M}\right)=x-\log M+\log\left(1+\frac{M}{e^x}\right)
\end{split}
\end{equation*}
and the corresponding formula for $\log(1+e^y/N)$, we may write \eqref{recuree1pm42fu1} as
\begin{equation}\label{recuree1pm42fu6}
\begin{split}
&\ -\log \binom{\mathfrak{m}+2\mathfrak{n}}{\mathfrak{m}} M!+e^x+(M-1)x-2\log Q_\mathfrak{m}(e^x)+\log \left(1+\frac{M}{e^x}\right)\\
=&\ -\log \binom{m+2n}{m} N!+e^y+(N-1)y-2\log Q_m(e^y)+\log \left(1+\frac{N}{e^y}\right)+\beta(k).
\end{split}
\end{equation}
We write $x=\log N+s$ and $y=\log N+t$ and note that, since $x>\log N$, we have $t>s-O(1)$ by Lemma~\ref{Lemma4}. We obtain
\begin{equation*}
\begin{split}
&\ -\log \binom{\mathfrak{m}+2\mathfrak{n}}{\mathfrak{m}} M!+Ne^s+(M-1)(\log N+s)-2\mathfrak{m}s+\log \left(1+\frac{M}{Ne^s}\right)\\
=&\ -\log \binom{m+2n}{m} N!+Ne^t+(N-1)(\log N+t)-2mt+\log \left(1+\frac{1}{e^t}\right)+O(1)
\end{split}
\end{equation*}
and hence, using Stirling's formula together with the estimates in \eqref{recuree1pm33fu3}, we have
\begin{equation*}
\begin{split}
e^t-e^s=&\ \frac{1}{N}\left(-\log \frac{\binom{\mathfrak{m}+2\mathfrak{n}}{\mathfrak{m}}}{\binom{m+2n}{m}}\frac{M!}{N!}+(M-N)\log N+\log\frac{M}{N}\right)\\
&\ +\left(\frac{M-2\mathfrak{m}}{N}-\frac{1}{N}\right)s-\left(1-\frac{2m+1}{N}\right)t+O(1)\\
\leq&\ \frac{M}{N}s+O(1)\leq Cs+O(1).
\end{split}
\end{equation*}
It follows that
\begin{equation}\label{recuree1pm42fu9}
\begin{split}
0<\phi(x)-x=t-s\leq e^{t-s}-1\leq Cse^{-s}+O(e^{-s})=O(1) \quad \text{for} \quad x>\log N.
\end{split}
\end{equation}
For $x>8\log N$ we have $s=x-\log N>3x/4$ and thus $2s/3>x/2$. It follows that $se^{-s}=O(e^{-2s/3})=O(e^{-x/2})$ and thus \eqref{recuree1pm42fu9} yields \eqref{recuree1pm36}. Noting that $s_{m,n}=\log N+O(1)$ by Lemma~\ref{Lemma3 trans} we can also deduce \eqref{recuree1pm38} from \eqref{recuree1pm42fu9}.

Since $g'_{\mathfrak{m},\mathfrak{n}}(x)=g'_{m,n}(\phi(x))\phi'(x)$ by the chain rule, \eqref{recuree1pm10} yields that
\begin{equation*}
\begin{split}
\phi'(x)=&\ \frac{\binom{m+2n}{m}}{\binom{\mathfrak{m}+2\mathfrak{n}}{\mathfrak{m}}}\frac{(m+2n)!}{(\mathfrak{m}+2\mathfrak{n})!}\frac{\left(\sum_{i=0}^{m}A_ie^{iy}\right)^2}{\left(\sum_{\mathfrak{i}=0}^{\mathfrak{m}}A_\mathfrak{i}e^{\mathfrak{i}x}\right)^2}e^{e^{x}+(\mathfrak{m}+2\mathfrak{n}+1)x-e^{y}-(m+2n+1)y}\\
=&\ \frac{\binom{m+2n}{m}}{\binom{\mathfrak{m}+2\mathfrak{n}}{\mathfrak{m}}}\frac{N!}{M!}\frac{\left(\sum_{i=0}^{m}A_ie^{iy}\right)^2}{\left(\sum_{\mathfrak{i}=0}^{\mathfrak{m}}A_\mathfrak{i}e^{\mathfrak{i}x}\right)^2}e^{e^{x}+Mx+\log M-e^{y}-Ny-\log N}.
\end{split}
\end{equation*}
Using \eqref{recuree1pm42fu6} and \eqref{recuree1pm42fu9} we obtain
\begin{equation*}
\begin{split}
\log \phi'(x)=y-x+\log\left(1+\frac{M}{e^x}\right)-\log\left(1+\frac{N}{e^y}\right)+\beta(k)
\end{split}
\end{equation*}
for $x>\log N$. Since $y>x-O(1)\geq \log N-O(1)$ by Lemma~\ref{Lemma4}, we have
\begin{equation*}
\begin{split}
\frac{N}{e^y}\leq e^{O(1)}\frac{M}{e^x}\leq e^{O(1)}\frac{M}{N}\leq O(1)+\beta(k).
\end{split}
\end{equation*}
Together with \eqref{recuree1pm36} and \eqref{recuree1pm38} and since for $A>0$ there exists $B>0$ such that $|t-1|\leq B|\log t|$ whenever $|\log t|\leq A$, the same arguments as the ones used before now yields \eqref{recuree1pm40}, as well as
\begin{equation}\label{recuree1pm42fu14}
\begin{split}
|\log \phi'(x)|=O(1) \quad \text{for} \quad x>\log N.
\end{split}
\end{equation}
Similarly, \eqref{recuree1pm42fu1} yields
\begin{equation*}
\begin{split}
\log \phi'(x)=\log \frac{M}{N}+\log\left(1+\frac{M}{e^x}\right)-\log\left(1+\frac{N}{e^y}\right)+\beta(k)
\end{split}
\end{equation*}
and hence, by the estimates for $R(e^x,M)$ and $R(e^y,N)$, that
\begin{equation*}
\begin{split}
\log \left(\frac{N}{M}\phi'(x)\right)=O\left(\frac{e^x}{M}\right) \quad \text{for} \quad x<\log N.
\end{split}
\end{equation*}
This yields \eqref{recuree1pm41} as well as $|\log \phi'(x)|=O(1)$ for $x<\log N$ which together with \eqref{recuree1pm42fu14} yields \eqref{recuree1pm42}.

\end{proof}

\subsection{Definition of a quasimeromorphic map}\label{Definition of a quasimeromorphic map} 

Now we begin to define a quasimeromorphic map by gluing functions $g_{m,n}$ with different pair of values $(m,2n)$ as in \cite[Section~3.2]{Bergweilereremenko2019}. By modifying the functions $g_{m,n}$ slightly, we obtain closely related functions $u_{m,n}$ and $v_{m,n}$ and then glue restrictions of these maps to half-strips along horizontal lines to obtain quasimeromorphic maps $U$ and $V$ which are defined in the right and left half-plane respectively. Then we will glue these functions along the imaginary axis to obtain a quasimeromorphic map $G$ in the plane. The following definitions of $U$ and $V$ are direct quotation from \cite[Section~3.2]{Bergweilereremenko2019}.

The maps $U$, $V$ and $G$ will commute with complex conjugation, so it will be enough to define them in the upper half-plane. We begin by constructing the map $U$. Instead of $g_{m,n}$ we consider the map
\begin{equation*}
\begin{split}
u_{m,n}: \{z\in\mathbb{C}: \Re z\geq0 \}\to\mathbb{C}, \quad u_{m,n}=g_{m,n}(z+s_{m,n}).
\end{split}
\end{equation*}
Note that $g_{m,n}$ is increasing on the real line, and maps $[0,\infty)$ onto $[2,\infty)$.

Let $(n_k)$ and $(m_k)$ be the two sequences from Lemma~\ref{Lemma1} and Lemma~\ref{Lemma2}, respectively and write $N_k=m_k+2n_k+1$. Put $U(z)=u_{m_k,n_k}$ in the half-strip
\begin{equation*}
\begin{split}
\Pi_k^{+}=\{x+iy: x>0, \quad  2\pi(k-1)<y<2\pi k\}.
\end{split}
\end{equation*}
The function $U$ will be discontinuous. In order to obtain a continuous function we consider the function $\psi_k:[0,\infty)\to[0,\infty)$ defined by $u_{m_{k+1},n_{k+1}}(x)=u_{m_k,n_k}(\psi_k(x))$. The function $\psi_k$ is closely related to the function $\phi$ considered in Lemma~\ref{Lemma4} and Lemma~\ref{Lemma5}. In fact, denoting by $\phi_k$ the function $\phi$ corresponding to $n_k$ and $n_{k+1}$, then
\begin{equation*}
\begin{split}
\psi_k(x)=\phi_k(x+s_{m_{k+1},n_{k+1}})-s_{m_{k},n_{k}}.
\end{split}
\end{equation*}
We then define $U:\{z\in\mathbb{C}: \Re z\geq 0\}\to \mathbb{C}$ by interpolating between $u_{m_{k+1},n_{k+1}}$ and $u_{m_{k},n_{k}}$ as follows: If $2\pi (k-1)\leq y<2\pi k$, say $y=2\pi (k-1)+2\pi t$ where $0\leq t<1$, then we put
\begin{equation*}
\begin{split}
U(x+iy)=u_{m_{k},n_{k}}((1-t)x+t\psi_k(x)+iy)=u_{m_{k},n_{k}}(x+iy+t(\psi_k(x)-x)).
\end{split}
\end{equation*}
Actually, by $2\pi i$-periodicity we have
\begin{equation*}
\begin{split}
U(x+iy)=u_{m_{k},n_{k}}((1-t)x+t\psi_k(x)+2\pi it).
\end{split}
\end{equation*}
The function $U$ defined this way is continuous in the right-half plane.

We now define a function $V$ in the left half-plane. In order to do so, we define
\begin{equation*}
\begin{split}
v_{m,n}:\{z\in\mathbb{C}: \Re z< 0\}\to \mathbb{C},  \quad v_{m,n}(z)=g_{m,n}\left(\frac{z}{m+2n+1}+s_{m,n}\right).
\end{split}
\end{equation*}
Note that $v_{m,n}$ maps $(-\infty,0]$ monotonically onto $(1,2]$.

Let $(n_k)$ and $(m_k)$ be as before and put $\mathcal{N}_k=\sum_{j=1}^{k}N_j$. This time we would like to define $V(z)=v_{m_k,n_k}(z)$ in the half-strip
\begin{equation*}
\begin{split}
\{x+iy: x<0, \ 2\pi \mathcal{N}_{k-1}\leq y<2\pi \mathcal{N}_k\},
\end{split}
\end{equation*}
but again this function would be discontinuous. In order to obtain a continuous function we again interpolate between $v_{m_{k+1},n_{k+1}}$ and $v_{m_k,n_k}$. Similarly as before we consider the map $\phi_k:(-\infty,0]\to (-\infty,0]$ defined by $v_{m_{k+1},n_{k+1}}(x)=v_{m_k,n_k}(\psi_k(x))$. Then we define $V:\{z\in \mathbb{C}:
\Re z\leq 0\}\to \mathbb{C}$ by interpolating between $v_{m_{k+1},n_{k+1}}$ and $v_{m_k,n_k}$ as follows: If $2\pi \mathcal{N}_{k-1}\leq y<2\pi \mathcal{N}_k$, say $y=2\pi \mathcal{N}_{k-1}+2\pi N_kt$ where $0\leq t<1$, then we put
\begin{equation*}
\begin{split}
V(x+iy)=v_{m_{k},n_{k}}((1-t)x+t\psi_k(x)+iy)=v_{m_{k},n_{k}}(x+iy+t(\psi_k(x)-x)).
\end{split}
\end{equation*}
This map $V$ is continuous in the left half-plane.

Now we define our map $G$ by gluing $U$ and $V$ along the imaginary axis. In order to do this we note that by construction we have $U(iy)=V(i(h(y)+\mathfrak{h}(y)))$ and thus $U(ig(y))=V(iy^{\gamma})$ for $y\geq 0$, with the maps $h$, $\mathfrak{h}$ and $g$ from Lemma~\ref{Lemma1} and Lemma~\ref{Lemma2}. Therefore, we will consider a homeomorphism $Q$ of the right half-plane
\begin{equation*}
\begin{split}
H^{+}=\{z\in\mathbb{C}:\Re z\geq 0\}
\end{split}
\end{equation*}
onto itself, satisfying $Q(\bar{z})=\overline{Q(z)}$, such that $Q(\pm iy)=\pm ig(y)$ for $y\geq 0$ while $Q(z)=z$ for $\Re z\geq 1$. We thus have to define $Q(z)$ in the strip $\{z\in\mathbb{C}:0<\Re z<1\}$. For $|\Im z|\geq 1$ we define $Q$ in this strip by interpolation; that is, we put
\begin{equation*}
\begin{split}
Q(x\pm iy)=x\pm i((1-x)g(y)+xy)  \quad
 \text{if} \quad 0<x<1 \quad \text{and} \quad y\geq 1.
\end{split}
\end{equation*}
In the remaining part of the strip we define $Q$ by
\begin{equation*}
\begin{split}
Q(z)=\left\{
       \begin{array}{ll}
         z|z|^{\gamma-1}, & \text{if} \quad 0<|z|<1; \\
         z,               & \text{if} \quad |z|>1, \quad \text{but} \quad  0<\Re z<1 \quad \text{and} \quad 0\leq |\Im z|\leq 1.
       \end{array}
     \right.
\end{split}
\end{equation*}
Note that $h(y)+\mathfrak{h}(y)=y$ for $0\leq y\leq 2\pi$ since $m_1=n_1=0$ implies that $N_1=1$, and thus $g(y)=y^{\gamma}$ for $0\leq y<1$. Thus $Q(\pm i y)=\pm ig(y)$ also for $|y|\leq 1$. Moreover, we have $g(1)=1$, meaning that the above expressions for $Q(z)$ do indeed coincide for $|\Im z|=1$. We conclude that the map $W=U\circ Q$ satisfies
\begin{equation*}
\begin{split}
W(\pm iy)=U(\pm ig(y))=V(\pm iy^{\gamma}) \quad \text{for} \quad y\geq 0.
\end{split}
\end{equation*}

Let now $\rho\in(1/2,1]$. We choose the constant $\gamma=1/(2\rho-1)$ and put $\sigma=\rho/(2\rho-1)$. We see that $\gamma\geq1$. Then the map
\begin{equation*}
\begin{split}
G(z)=\left\{
       \begin{array}{ll}
         W(z^{\rho}),         & \text{if} \quad |\arg z|\leq \frac{\pi}{2\rho}; \\
         V(-(-z)^{\sigma}),   & \text{if} \quad |\arg(-z)|\leq \frac{\pi}{2\sigma},
       \end{array}
      \right.
\end{split}
\end{equation*}
is continuous in $\mathbb{C}$. Here, for $\mu>0$, we denote by $z^{\mu}$ the principal branch of the power which is defined in $\mathbb{C}\setminus(-\infty,0]$.

\subsection{Estimation of the dilatation}\label{Estimation of the dilatation} 

With the definition in subsection~\ref{Definition of a quasimeromorphic map}, we now estimate the dilatation. After making some changes to the notations, the process will be the same as that in \cite[Section~3.3]{Bergweilereremenko2019}, except that we have a slightly different error term when applying Lemma~\ref{Lemma5}.

In the case $\gamma>1$, together with the estimates in \eqref{recuree1pm71}--\eqref{recuree1pm109}, we may replace each error term $O(1/k^{\delta})$, where $\delta>0$ is a constant, appearing in \cite[Section~3.3]{Bergweilereremenko2019} by $O(1/(\log(k+2))^2)$. Then we may replace any sum of the type $\sum_{k=1}^{\infty}1/k^{1+\delta}$ by a sum of the type $\sum_{k=1}^{\infty}1/k(\log(k+2))^2$, which also converges.

In the case $\gamma=1$, we have $n_1=0$ and $n_k=3$ for all $k\geq 1$. By the comments below Theorem~\ref{maintheorem2}, we only need to consider the case $0<\delta\gamma<1$. In this case the notation $\beta(k)$ used in the proof of Lemma~\ref{Lemma3}--Lemma~\ref{Lemma5} in subsection~\ref{Preliminary lemmas} shall be replaced by a bounded term $O(1)$. Now, since $m_k=0$ or $m_k=1$, it is possible that $M>N$, $M=N$ or $M<N$. In all these three cases, it is easy to check that the estimates in Lemma~\ref{Lemma3} and Lemma~\ref{Lemma3 trans} hold with the new definition of $\beta(k)$. We note that the increasing diffeomorphism $\phi$ in \eqref{diffeo} is just the identity when $M=N$. Further, when $\mathfrak{m}>m$ or $\mathfrak{m}<m$, by the equation \eqref{recuree1pm32fu2} we easily show that the fixed point of $\phi(x)$, say $p$, satisfies $p=O(1)$. Recall from Lemma~\ref{Lemma2} that the subsequence $(m_{k_i})$ such that $m_{k_i}=1$ satisfies $k_{i}\geq c i^{1/\delta\gamma}$ for a fixed positive constant $c$. Then, after making some slightly modifications to Lemma~\ref{Lemma5}, we may replace any sum of the type $\sum_{k=1}^{\infty}1/k^{1+\delta}$ in \cite[Section~3.3]{Bergweilereremenko2019} by a sum of the type $\sum_{i=1}^{\infty}1/k_i$, which also converges.

By the above discussions, below we only recall some notations that will be used in the following proof and the calculations will be omitted. For a quasimeromorphic map $f$, let
\begin{equation*}
\begin{split}
\mu_f(z)=\frac{f_{\bar{z}}(z)}{f_z(z)} \quad \text{and} \quad K_f(z)=\frac{1+|\mu_f(z)|}{1-|\mu_f(z)|}.
\end{split}
\end{equation*}
We have to estimate $K_G(z)-1$. Write
\begin{equation}\label{recuree1pm59}
\begin{split}
q(x+iy)=x+iy+t(\psi_k(x)-x)=x+iy+\left(\frac{y}{2\pi}-(k-1)\right)(\psi_k(x)-x).
\end{split}
\end{equation}
Thus
\begin{equation}\label{recuree1pm60}
\begin{split}
q_z(z)=1+a(z)-ib(z) \quad \text{and} \quad q_{\bar{z}}=a(z)+ib(z),
\end{split}
\end{equation}
with
\begin{equation}\label{recuree1pm61}
\begin{split}
a(x+iy)=\frac{t}{2}(\psi_k'(x)-1) \quad \text{and} \quad b(x+iy)=\frac{1}{4\pi}(\psi_k(x)-x).
\end{split}
\end{equation}
Then from the proof in \cite[Section~3.3]{Bergweilereremenko2019} we deduce that
\begin{equation}\label{recuree1pm69}
\begin{split}
K_U(z)-1 \leq \frac{4(1+r(x))r(x)}{\min\{1,\psi_k'(x)\}},
\end{split}
\end{equation}
where
\begin{equation}\label{recuree1pm68}
\begin{split}
r(x)=|\psi_k'(x)-1|+|\psi_k(x)-x|.
\end{split}
\end{equation}
Then we can show that $K_U(z)-1<\infty$, which together with \eqref{recuree1pm36} and \eqref{recuree1pm42} shows that $U$ is quasimeromorphic in the right half-plane $H^+$. Moreover, by using \eqref{recuree1pm16} and \eqref{recuree1pm17} and $g'(y)=\gamma x^{\gamma-1}/N$ we may show that the function $Q$ defined in subsection~\ref{Definition of a quasimeromorphic map} is quasiconfromal so that $W=U\circ Q$ is quasimeromorphic in $H^+$.

On the other hand, to obtain the estimate of $K_G(z)$ for $\arg(-z)<\pi/(2\sigma)$, we have $V(z)=V_{m_k,n_k}(q(z))$ where, instead of \eqref{recuree1pm59}, we have
\begin{equation*}
\begin{split}
q(x+iy)=x+iy+\frac{1}{N_k}\left(\frac{y}{2\pi}-\mathcal{N}_{k-1}\right)\left(\psi_k(x)-x\right)
\end{split}
\end{equation*}
for $2\pi \mathcal{N}_{k-1}\leq y <2\pi \mathcal{N}_k$, with
\begin{equation*}
\begin{split}
\psi_k(x)=N_k\phi\left(\frac{x}{N_{k+1}}+s_{m_{k+1}.n_{k+1}}\right)-N_ks_{m_k,n_k}.
\end{split}
\end{equation*}
Now \eqref{recuree1pm60} holds with $a(x+iy)$ as in \eqref{recuree1pm61}, but
\begin{equation*}
\begin{split}
b(x+iy)=\frac{1}{4\pi N_k}\left(N_k\phi_k\left(\frac{x}{N_{k+1}}\right)-x\right).
\end{split}
\end{equation*}
Instead of \eqref{recuree1pm68} and \eqref{recuree1pm69} we obtain
\begin{equation*}
\begin{split}
K_V(z)-1 &\leq \frac{4(1+r(x))r(x)}{\min\{1,\phi_k'(x)\}}
\end{split}
\end{equation*}
with
\begin{equation*}
\begin{split}
r(x)=\left|\psi'_k(x)-1\right|+\frac{1}{N_k}|\psi_k(x)-x|.
\end{split}
\end{equation*}
Then, as in \cite[Section~3.3]{Bergweilereremenko2019} we can finally show that, if $r>0$,
\begin{equation*}
\begin{split}
\int_{|z|>r}\frac{K_G(z)-1}{x^2+y^2}dxdy <\infty.
\end{split}
\end{equation*}
Thus, in both of the two cases $\gamma>1$ and $\gamma=1$, $G$ satisfies the hypothesis of the Teichm\"uller--Wittich--Belinskii theorem \cite[\S~V.6]{lehtoVirtanen2008}. This theorem, together with the existence theorem for quasiconformal mappings \cite[\S~V.1]{lehtoVirtanen2008}, yields that there exists a quasiconformal homeomorphism $\tau:\mathbb{C}\to\mathbb{C}$ and a meromorphic function $F$ such that
\begin{equation}\label{recuree1pm117}
\begin{split}
G(z)=F(\tau(z)) \quad \text{and} \quad \tau(z)\sim z \quad \text{as} \quad z\to\infty.
\end{split}
\end{equation}

\subsection{Completion of the proof}\label{Completion of the proof} 

We now begin to estimate the counting functions of the zeros and poles of $F$, respectively. For two positive quantities $S_1(r)$ and $S_2(r)$, below we shall use the notation $S_1(r)\asymp S_2(r)$, provided that there are two positive constants $d_1$ and $d_2$ such that $d_2 S_2(r)\leq S_1(r)\leq d_1S_2(r)$.

Let $r>0$ and choose $k\in \mathbb{N}$ such that $2\pi(k-1)<r\leq 2\pi k$. It follows from the construction of $(n_k)$ and $(m_k)$ and \eqref{recuree1pm109} that
\begin{equation}\label{recuree1pm118}
\begin{split}
n(r,0,U)\leq 4\sum_{j=1}^{k}n_j \sim 2\mathcal{N}_k \sim 2(2\pi)^{\gamma-1}k^{\gamma}\leq d_1r^{\gamma}
\end{split}
\end{equation}
for some positive constant $d_1$. On the other hand, we also have
\begin{equation*}
\begin{split}
n(r,0,U)\geq \sum_{j=1}^{k}n_j \sim \frac{1}{2}\mathcal{N}_k \sim \frac{1}{2}(2\pi)^{\gamma-1}k^{\gamma}\geq d_2r^{\gamma}
\end{split}
\end{equation*}
for some positive constant $d_2$. Therefore, $n(r,0,U)\asymp r^{\gamma}$. Similarly, if $2\pi\mathcal{N}_{k-1} <r \leq 2\pi \mathcal{N}_k$, then
\begin{equation*}
\begin{split}
n(r,0,V)\leq 4\sum_{j=1}^{k}n_j \sim 2\mathcal{N}_k \sim 2(2\pi)^{\gamma-1}k^{\gamma}\leq d_3r
\end{split}
\end{equation*}
for some positive constant $d_3$ and also that
\begin{equation*}
\begin{split}
n(r,0,V)\geq \sum_{j=1}^{k}n_j \sim \frac{1}{2}\mathcal{N}_k \sim \frac{1}{2}(2\pi)^{\gamma-1}k^{\gamma}\geq d_4r
\end{split}
\end{equation*}
for some constant $d_4$. Therefore, $n(r,0,V)\asymp r$. Since $\sigma=\rho\gamma$, we conclude that $n(r,0,G)\asymp r^{\rho\gamma}$. Now \eqref{recuree1pm117} yields $n(r,0,F)\asymp r^{\rho\gamma}$ and hence $N(r,0,F)\asymp r^{\rho\gamma}$ as $r\to\infty$.

Recall that the function $g(x)$ in Lemma~\ref{Lemma1} satisfies $g(x)=x(1+o(1))$ as $x\to\infty$ and also that $x(1+o(1))/3\leq g(x)\leq x(1+o(1))$ as $x\to\infty$ when $\gamma=1$. When $1<\delta\gamma\leq \gamma$, by the same arguments as above we have $N(r,\infty,F)\asymp (\log r)^{-2}r^{\delta\rho\gamma}$; when $0\leq \delta\gamma\leq 1$, similarly as for equation \eqref{recuree1pm118} we have $N(r,\infty,F)\leq d_5r^{\delta\rho\gamma}$ for a positive constant $d_5$ and, moreover, there is an infinite sequence $(r_n)$ such that $N(r_n,\infty,F)\geq d_6r_n^{\delta\rho\gamma}$ for a constant $d_6$.

Next we note that there is a set $\Omega_1$ with finite linear measure such that, for all $|z|\not\in \Omega_1$, the functions $P_n(e^{z})$ and $Q_{n}(e^z)$ satisfy
\begin{equation*}
\left\{
\begin{array}{ll}
\exp(-(2n+\varepsilon)r)\leq\left|P_n(e^{z})\right|\leq \exp((2n+\varepsilon)r), &\\
\exp(-(m+\varepsilon)r)\leq\left|Q_m(e^{z})\right|\leq \exp((m+\varepsilon)r), &
\end{array}
\right.
\end{equation*}
respectively, where $\varepsilon>0$ is a small constant. This implies that
\begin{equation*}
\begin{split}
|g_{m,n}(z)|\leq \exp((m+\varepsilon)r+(2n+\varepsilon)r)\left|\exp(e^z)\right|
\end{split}
\end{equation*}
for all $z\in \mathbb{C}$ such that $|z|\not\in \Omega_1$. For simplicity, we denote $L=\max\{m,2n\}$ and $L_k=\max\{m_k,2n_k\}$. Clearly this implies that there is a set $\Omega_2$ of finite linear measure such that
\begin{equation*}
\left\{
\begin{array}{ll}
|v_{m,n}(z)|\leq \exp((L+\varepsilon)r)\left|\exp(e^z)\right| \quad \text{for} \quad z\in H^{-} \quad \text{and} \quad |z|\not\in \Omega_2, &\\
|u_{m,n}(z)|\leq \exp((L+\varepsilon)r)\left|\exp(e^z)\right| \quad \text{for} \quad z\in H^{+} \quad \text{and} \quad |z|\not\in \Omega_2. &
\end{array}
\right.
\end{equation*}

Let $x+iy\in H^{-}$ with $\Im z\geq 0$ and $k\in \mathbb{N}$ with $2\pi \mathcal{N}_{k-1}\leq \Im z<2\pi \mathcal{N}_k$. With $t=(y-2\pi \mathcal{N}_{k-1})/{2\pi N_k}$ we have $0\leq t<1$ and
\begin{equation*}
\begin{split}
|V(x+iy)|=|v_{m_k,n_k}((1-t)x+t\psi_k(x)+iy)|\leq \exp(L_k+\varepsilon)r).
\end{split}
\end{equation*}
Thus
\begin{equation}\label{recuree1pm125}
\begin{split}
|V(z)|\leq \exp((L_k+\varepsilon)r),
\end{split}
\end{equation}
for $z\in H^{-}$ and $|z|\not\in \Omega_2$. For $z\in H^{+}$ and $|z|\not\in \Omega_2$ we have
\begin{equation}\label{recuree1pm126}
\begin{split}
|g_{m,n}(z)|\leq g_{m,n}(\Re z)\leq \exp((L+\varepsilon)r)\exp\left(e^{\Re z}\right).
\end{split}
\end{equation}
Again, assuming that $\Im z\geq 0$, we choose $k\in \mathbb{N}$ such that $2\pi (k-1)\leq \Im z <2\pi k$. Then
\begin{equation*}
\begin{split}
|U(z)|=|u_{m_k,n_k}(q(z))|=|g_{m_k,n_k}(q(z)+s_{m_k,n_k})|\leq \exp((L_k+\varepsilon)r)\exp\left(e^{\Re z}\right),
\end{split}
\end{equation*}
where $q(z)$ is defined by \eqref{recuree1pm59}. Noting that $m_k+2n_k=O(k^{\gamma-1})=O(|z|^{\gamma-1})$ by \eqref{recuree1pm18} and thus, by Lemma~\ref{Lemma3 trans},
\begin{equation}\label{recuree1pm128}
\begin{split}
s_{m_k,n_k}=\log N_k+O(1)=\log (m_k+2n_k)+O(1)=O(\log |z|),
\end{split}
\end{equation}
we deduce from \eqref{recuree1pm36} and \eqref{recuree1pm126} that
\begin{equation*}
\begin{split}
\log|U(z)| \leq&\ (L_k+\varepsilon)r+\exp(\Re q(z)+s_{m_k,n_k})\\
\leq&\ \exp((1+o(1))z)+O(|z|^{\gamma})\leq \exp((1+o(1))z).
\end{split}
\end{equation*}
Together with \eqref{recuree1pm125} we conclude that
\begin{equation*}
\begin{split}
\log|G(z)|\leq \exp((1+o(1))z^{\rho})
\end{split}
\end{equation*}
as $|z|\to\infty$ outside the set $\Omega_2$. Hence \eqref{recuree1pm117} yields that
\begin{equation*}
\begin{split}
\log|F(z)|\leq \exp((1+o(1))z^{\rho})
\end{split}
\end{equation*}
as $|z|\to\infty$ outside an exceptional set $\Omega_3$ of finite linear measure. Then, together with \eqref{recuree1pm133}, the lemma on the logarithmic derivative \cite{Hayman1964Meromorphic,Laine1993} now implies that $E=F/F'$ satisfies
\begin{equation*}
\begin{split}
m\left(r,\frac{1}{E}\right)=O(r^{\rho}),
\end{split}
\end{equation*}
outside a set of finite linear measure. By the relation $N(r,\infty,F)\leq d_5(\log r)^{-2}r^{\delta\rho\gamma}$ and $N(r,0,F)\asymp r^{\rho\gamma}$, and since the zeros and poles of $F$ are simple, we have
\begin{equation}\label{recuree1pm133}
\begin{split}
N\left(r,0,E\right)\asymp r^{\rho\gamma}.
\end{split}
\end{equation}
We conclude that
\begin{equation*}
\begin{split}
T\left(r,E\right)\asymp r^{\rho\gamma}.
\end{split}
\end{equation*}
In particular, $E$ has finite order. $E$ is clearly a Bank--Laine functions and thus $E$ is the product of two solutions of \eqref{bank-laine0}. The lemma on the logarithmic derivative, together with \eqref{recuree1pm7} and \eqref{recuree1pm133} also implies that
\begin{equation*}
\begin{split}
m(r,A)=2m\left(r,\frac{1}{E}\right)+O(\log r)=O(r^{\rho}).
\end{split}
\end{equation*}
We thus have $\lambda(E)=\rho(E)=\rho\gamma$ and $\rho(A)\leq \rho$. Since $1/\rho+1/\rho\gamma=2$ by the definition of $\gamma$, we deduce from \eqref{bank-laine1} that actually $\rho(A)=\rho$. This completes the proof of Theorem~\ref{maintheorem2}.

\section{Proof of Theorem~\ref{maintheorem3}}\label{Proof of maintheorem3} 

Recall from the introduction that $A_i=\frac{2n!}{(i+2n)!}$, $i=1,\cdots,m$ and $B_j=\frac{(-1)^jm!}{(m+j)!}$, $j=1,\cdots,2n$. We note that for $|z|>B_{2n-1}/B_{2n}=m+2n$,
\begin{equation}\label{recuree1pm137}
\begin{split}
\left|P_n(z)-B_{2n}z^{2n}\right|=&\ \left|\sum_{j=0}^{2n-1}B_jz^j\right|\leq\sum_{j=0}^{2n-1}\left|B_j\right||z|^j=\sum_{j=0}^{2n-1}\frac{\left|B_j\right||z|^{2n-1}}{|z|^{2n-1-j}}\\
&\ \leq \left|B_{2n-1}\right||z|^{2n-1}\sum_{j=0}^{2n-1}\frac{1}{2^{2n-1-j}}\leq \frac{B_{2n-1}}{B_{2n}|z|}(B_{2n}|z|^{2n})
\end{split}
\end{equation}
and similarly
\begin{equation}\label{recuree1pm138}
\begin{split}
\left|Q_m(z)-A_mz^m\right| \leq \frac{A_{m-1}}{A_{m}|z|}(A_{m}|z|^{m})
\end{split}
\end{equation}
and thus in particular $R_{m,n}(z)\not=0,\infty$ for $|z|>m+2n$. Then we can follow the way as in \cite[Section~3.4]{Bergweilereremenko2019} to estimate the asymptotics of $F$, $E$ and $A$.

Let $0<\varepsilon_1<\varepsilon_2<\varepsilon_3<\varepsilon<1$ and, for $l\in\{1,2\}$, put
\begin{equation*}
\begin{split}
H_{\varepsilon_l}^{+}=\left\{z\in\mathbb{C}: \quad |\arg z|\leq (1-\varepsilon_l)\frac{\pi}{2}  \right\}.
\end{split}
\end{equation*}
Given $z\in H_{\varepsilon_1}^{+}$ with $\Im z>0$, we choose $k\in \mathbb{N}$ with $2\pi (k-1)\leq \Im z<2\pi k$. We then have $k=O(\Re z)$ and hence, by \eqref{recuree1pm18}, $\log (m_k+2n_k)=o(\Re z)$ and $m_k+2n_k=o(e^{z})$ as $z\to\infty$ in $H_{\varepsilon_1}^{+}$. We deduce from \eqref{recuree1pm137} and \eqref{recuree1pm138} that
\begin{equation*}
\begin{split}
R_{m_k,n_k}(e^z) \sim \frac{B_{2n_k}}{A_{m_k}}e^{(2n_k-m_k)z}.
\end{split}
\end{equation*}
Recall that $A_{m_k}B_{2n_k}=\frac{m_k!(2n_k)!}{[(m_k+2n_k)!]^2}$. By looking at the proof of Lemma~\ref{Lemma1} and Lemma~\ref{Lemma2}, and together with Stirling's formula, we have
\begin{equation*}
\begin{split}
\log R_{m_k,n_k}(e^z)\sim2\log B_{2n_k}-\log (A_{m_k}B_{2n_k})+(2n_k-m_k)z\sim (2n_k-m_k)z.
\end{split}
\end{equation*}
It follows that
\begin{equation*}
\begin{split}
\log g_{m_k,n_k}(z)=(2n_k-m_k)(1+o(1))z+e^z \sim e^z \quad \text{as} \quad z\to\infty \quad \text{in} \quad H_{\varepsilon_1}^+.
\end{split}
\end{equation*}
With $q(z)$ as in \eqref{recuree1pm59} we have
\begin{equation*}
\begin{split}
\log U(z)=\log u_{m_k,n_k}(q(z))=\log g_{m_k,n_k}(q(z)+s_{m_k,n_k})
\end{split}
\end{equation*}
and $q(z)+s_{m_k,n_k}\sim z$ as $z\to\infty$ in $H_{\varepsilon_2}^+$ by \eqref{recuree1pm128}. Thus $q(z)+s_{m_k,n_k}\in H_{\varepsilon_1}^+$ if $z\in H_{\varepsilon_2}^+$ and $|z|$ is sufficiently large and we deduce from the last two equations that
\begin{equation*}
\begin{split}
\log U(z) \sim \exp(q(z)+s_{m_k,n_k})=\exp((1+o(1))z) \quad \text{as} \quad z\to\infty \quad \text{in} \quad H_{\varepsilon_2}^+
\end{split}
\end{equation*}
and thus
\begin{equation*}
\begin{split}
\log G(z)=\log U(z^{\rho})=\exp((1+o(1))z^{\rho}) \quad \text{for} \quad |\arg z|<(1-\varepsilon_2)\frac{\pi}{2\rho}.
\end{split}
\end{equation*}
This implies that
\begin{equation*}
\begin{split}
\log F(z)=\log G(\tau^{-1}(z))=\exp((1+o(1))z^{\rho}) \quad \text{for} \quad |\arg z|<(1-\varepsilon_3)\frac{\pi}{2\rho}
\end{split}
\end{equation*}
and hence
\begin{equation}\label{recuree1pm147}
\begin{split}
\log\log F(z)\sim z^{\rho} \quad \text{for} \quad |\arg z|<(1-\varepsilon_3)\frac{\pi}{2\rho}
\end{split}
\end{equation}
as $|z|\to\infty$. An asymptotic equality like \eqref{recuree1pm147} can be differentiated by passing to a smaller sector. If $|\arg z|<(1-\varepsilon_4)\pi/(2\rho)$, then $\{\zeta: |\zeta-z|=c|z|\}$ is contained in $\{\zeta: |\arg z|<(1-\varepsilon_3)\pi/(2\rho)\}$ for sufficiently small $c$. Thus we can obtain \eqref{recuree1pm3 fu2} by following exactly the same process as in \cite[Section~3.4]{Bergweilereremenko2019}. We omit those details.

The proof of \eqref{recuree1pm3 fu3} and \eqref{recuree1pm3 fu4} is similar. Note that $N=m+2n+1$. From the proof of Lemma~\ref{Lemma3} we have
\begin{equation*}
\begin{split}
h_{m,n}(z)-1= \frac{1}{\binom{m+2n}{m}N!}\frac{e^zz^{N}}{Q_m(z)^2}\left(1-U(z)\right),
\end{split}
\end{equation*}
where $U(z)$ satisfies
\begin{equation*}
\begin{split}
|U(z)|\leq \left(1+\frac{m}{1+2n}\right)\left(\left|\frac{z}{z+N-2m}\right|+\left|\frac{12z}{(z+N-2m)^2}\right|\right).
\end{split}
\end{equation*}
Then, by the estimates in \eqref{recuree1pm71} and \eqref{recuree1pm72}, we have
\begin{equation*}\label{recuree1pm155}
\begin{split}
h_{m,n}(z)=1+\frac{z^{N}}{\binom{m+2n}{m}N!}+O(z^{N+1})=1+(1+O(z))\frac{z^{N}}{\binom{m+2n}{m}N!}
\end{split}
\end{equation*}
as $z\to0$. Thus, by Lemma~\ref{Lemma3 trans} together with Stirling's formula we have
\begin{equation*}
\begin{split}
\log(v_{m,n}-1)&=\log \left(g_{m,n}\left(\frac{z}{N}+s_{m,n}\right)-1\right)=\log \left(h_{m,n}\left(e^{\frac{z}{N}+s_{m,n}}\right)-1\right)\\
&=-\log\binom{m+2n}{m}N!+z+Ns_{m,n}+\log(1+O(e^{\frac{z}{N}}))\\
&=z+\log(1+O(e^{\frac{z}{N}}))+O(N\log N).
\end{split}
\end{equation*}
Similarly as before, for $l\in\{1,2\}$, we consider the sectors
\begin{equation*}
\begin{split}
H_{\varepsilon_l}^{-}=\left\{z\in\mathbb{C}: \quad |\arg (-z)|\leq (1-\varepsilon_l)\frac{\pi}{2}  \right\}.
\end{split}
\end{equation*}
For $z\in H_{\varepsilon_1}^{-}$ with $\Im z>0$ we choose $k\in \mathbb{N}$ with $2\pi\mathcal{N}_{k-1}\leq
\Im z<2\pi \mathcal{N}_k$. We can deduce from \eqref{recuree1pm18} and \eqref{recuree1pm109} that
$\Re z/N_{k-1}+s_{m_k,n_k}\to-\infty$ as $z\to\infty$ in $z\in H_{\varepsilon_1}^{-}$. This implies that $e^{z/N_{k-1}+s_{m_k,n_k}}\to 0$ and hence $O(e^{z/N_{k-1}+s_{m_k,n_k}})\to 0$ as $z\to\infty$ in $z\in H_{\varepsilon_1}^{-}$. Moreover, $N_k\log N_k=o(|z|)$ as $z\to\infty$ in $z\in H_{\varepsilon_1}^{-}$, again by \eqref{recuree1pm72} and \eqref{recuree1pm109}. It follows that $\log(v_{m_k,n_k}-1)\to z$ as $z\to\infty$ in $z\in H_{\varepsilon_1}^{-}$. Then we can obtain \eqref{recuree1pm3 fu4} by following exactly the same process as in \cite[Section~3.4]{Bergweilereremenko2019}. We omit those details.

\end{document}